\documentclass[12pt,a4paper]{article}
\usepackage{amsmath,amssymb,amsthm,amsfonts,amscd,euscript,verbatim, t1enc, newlfont}
\usepackage{hyperref}
\usepackage{graphicx}
\usepackage[all]{xy}
 \theoremstyle{plain}
\newtheorem{theo}{Theorem}[subsection]

\newtheorem{theore}{Theorem}[section]
\newtheorem{pr}[theo]{Proposition}

 \newtheorem{lem}[theo]{Lemma}
 \newtheorem{coro}[theo]{Corollary}
  
\theoremstyle{remark}
\newtheorem{rema}[theo]{Remark}

\newtheorem{rrema}{Remark}[section]
\theoremstyle{definition}
\newtheorem{defi}[theo]{Definition}

 \newcommand\lan{\langle}
\newcommand\ra{\rangle}

\newcommand\ob{^{-1}}

\newcommand\dmge{DM^{eff}_{gm}{}}

\newcommand\dmges{DM_{gm}^{eff}(U)}
\newcommand\dmes{DM^{eff}(U)}

\newcommand\cors{Cor(U)}
\newcommand\birs{DM_{gm}^{o}(U)}
\newcommand\birchs{\chow^o(U)}

\newcommand\obj{\operatorname{Obj}}
\newcommand\mo{\operatorname{Mor}}
\newcommand\id{\operatorname{id}}
\newcommand\isom{\operatorname{Isom}}

\newcommand\fgf{\mathcal{FGF}}

\newcommand\cu{\underline{C}}
\newcommand\du{\underline{D}}
\newcommand\ddu{\underline{\mathcal{D}}}

\newcommand\eu{\underline{E}}
\newcommand\eeu{\underline{\mathcal{E}}}

\newcommand\au{\underline{A}}
\newcommand\aau{\mathfrak{A}}

\newcommand\hu{\underline{H}}

\newcommand\hrt{{\underline{Ht}}}
\newcommand\hw{{\underline{Hw}}}

\newcommand\n{\mathbb{N}}
\newcommand\z{{\mathbb{Z}}}

\newcommand\af{\mathbb{A}}

\newcommand\rr{\mathcal{R}}

\newcommand\ns{\{0\}}

\DeclareMathOperator\inli{\varinjlim}

\newcommand\chow{Chow}

\newcommand\ab{Ab}

\DeclareMathOperator\co{\operatorname{Cone}}
\DeclareMathOperator\enom{\operatorname{End}}
\DeclareMathOperator\kar{\operatorname{Kar}}
\DeclareMathOperator\adfu{\operatorname{AddFun}}

\DeclareMathOperator\extr{\operatorname{Ext}}

\DeclareMathOperator\cha{\operatorname{char}}

\newcommand\w{{\mathfrak{w}}}

\newcommand\gd{\mathfrak{D}}

\newcommand\dmcs{DM_c(U)}
\newcommand\sss{{\mathcal{S}}}

\begin{document}

  \title{Non-commutative localizations of additive categories and weight structures; applications to birational motives} 
 \author{Mikhail V. Bondarko, Vladimir A. Sosnilo
 \thanks{ 
  The first author was supported  supported by RFBR
(grant no.  14-01-00393).  
The second author was supported by the Chebyshev Laboratory  (Department of
Mathematics and Mechanics, St. Petersburg State University)  under RF Government
grant 11.G34.31.0026 and by JSC "Gazprom Neft".} }
 \maketitle 
  
	\let\thefootnote\relax\footnote{ MSC2010 Primary: 18E35, 16S85, 14C15; Secondary: 18E30, 18E05, 14F42, 19E15, 16U20, 18G35, 16E20.}

\begin{abstract} 
In this paper we demonstrate that 
{\it non-commutative localizations}  
of arbitrary additive categories (generalizing those defined by Cohn in the setting of rings)
are closely (and naturally) related to weight structures.  Localizing an
arbitrary triangulated category $\cu$ by a set $S$ of  morphisms in the heart
$\hw$ of a weight structure $w$ on it one obtains a triangulated category endowed with
a weight structure $w'$.  The heart of $w'$ is a certain version of the idempotent completion
of the non-commutative localization  $\hw[S\ob]_{add}$   (of  $\hw$ by $S$).
The functor $\hw\to \hw[S\ob]_{add}$ is the
natural categorical version of Cohn's localization of a ring, i.e., is universal
among additive functors that make all elements of $S$ invertible. 
For any additive category $\au$, taking $\cu=K^b(\au)$  we obtain a very efficient
tool for computing $\au[S\ob]_{add}$; using it, we 
generalize the calculations of Gerasimov
and Malcolmson (made for rings only). 
 We also prove that  $\au[S\ob]_{add}$ coincides with the "abstract"
localization $\au[S\ob]$ (as constructed by Gabriel and Zisman) if $S$ contains
all identity morphisms of $\au$ and is closed with respect to direct sums.

We apply our results to certain categories of birational motives $\birs$
(generalizing those defined by Kahn and Sujatha). We define $\birs$ for
arbitrary $U$ as a certain localization of $K^b(\cors)$ and obtain a weight structure for
it. 
When $U$ is the spectrum of a perfect field, the weight structure obtained
is compatible with the corresponding Chow and Gersten weight structures defined by the
first author in previous papers.
For a general $U$ the result is completely new. It enables us to calculate the
Grothendieck group of $\birs$. 
The existence of the corresponding adjacent $t$-structure is also a new result over a general base scheme; its heart is a certain category of birational sheaves with transfers over $U$.
\end{abstract}

\tableofcontents

 \section*{Introduction}
 
 For an additive category $\au$ and some set $S\subset \mo(\au)$ of morphisms of
 $\au$ it is natural to ask the following questions:
 
 (i) does there
  exist an initial object in the category of those additive functors from $\au$
  to an additive category
  $\au'$ such that all the elements of $S$ become invertible in $\au'$? 
  
  (ii) provided that the initial object does exist, how can one describe its target $\au[S^{-1}]_{add}$?
  
  (iii) is $\au[S^{-1}]_{add}$ isomorphic to the ``abstract'' (i.e.,
  ``Gabriel-Zisman'') localization of $\au$ by $S$?
  
  In the case when $\au$ is the category of finitely generated free (left)
  modules over a associative unital ring $R$, the answers to questions (i) and
  (ii) are given by the theory of non-commutative localizations of rings
  introduced in \cite{cohn} (see \S\ref{scompco} and \S\ref{sadja} below for
  more references). In the current paper we extend the corresponding results to
  a 
  more general setting of an arbitrary additive category $\au$. We
  also recall that  the answer to (iii) is positive if $S$ contains all
  identity morphisms of $\au$ and is closed under 
  direct sums (this certainly implies that the answer to (i) is positive in general).
  
   We give a comprehensive description of $\au[S^{-1}]_{add}$. Our main
   ``computational'' tool 
   is the full embedding of $\au[S^{-1}]_{add}$ into the Verdier localization of
   $K^b(\au)$ (see \S\ref{snotata} below) by $S$. So we prove the following result. 
   
   \begin{theore} \label{taddlocu}
   
   Let $\au$ be a (small) 
additive category and let $S$ be an arbitrary
 subset of $\mo(\au)$.
   
   Denote by $\du$ the triangulated subcategory of $K^b(\au)$ 
   generated 
   (see \S\ref{snotata})
    by the cones of elements of $S$; denote by $l$ the Verdier localization functor $K^b(\au)\to \eu=K^b(\au)/\du$. 
    Denote by
$\au[S^{-1}]_{add}\subset \eu$ the full subcategory whose objects are those of $\au$; denote   
       by $u$ the functor $\au\to \au[S^{-1}]_{add}$ induced by $l$.

    Then 
    $u$ is characterized by the following universal property:  an additive functor $F:\au\to \au'$  factorizes through $u$ if and only if $F$ converts all elements of $S$ into isomorphisms; if it does, then such a factorization is unique. 
\end{theore}

   Our methods (of proving this theorem) are closely related to weight structures (as was shown in
   several previous papers of the first author, weight structures are important
   cousins of $t$-structures; cf. also Remark \ref{rstws}(4) below). We prove that a
   weight structure $w$ on a triangulated category $\cu$ necessarily induces a
   weight structure on $\cu/\du$ if a triangulated subcategory $\du$ of $\cu$ is
   generated by a set of objects of ``$w$-length $1$'', i.e., 
   by cones of some set of morphisms between objects of the heart $\hw$
   of a weight structure $w$ on $\cu$ 
   (in contrast to the setting of $t$-structures, $w$ does
   not have to induce a weight structure on $\du$; we will say more on this
   matter at the end of this Introduction). The heart of this weight structure
   is the Karoubi-closure of $\hw[S\ob]_{add}$ in $\cu/\du$. So, one can
   describe $\au[S\ob]_{add}$ using any triangulated category $\cu$ with a weight structure whose heart is (contained in) the
   idempotent completion of $\au$. 
We also prove 
for a 
triangulated subcategory $\ddu\subset\cu$ 
that is (``compactly'') generated by
$\{\co (S)\}$ {\it as a localizing subcategory} 
that $\cu/\ddu$ possesses
a $t$-structure {\it adjacent} to $w_{\cu/\ddu}$ and calculate its heart.

We also apply our results to  certain triangulated categories of (geometric)
birational motives. Those are obtained from (a version of) Voevodsky's effective
geometric motives over a scheme $U$ by inverting birational equivalences
(and by Karoubization). 
So, there exists a weight structure $w_{bir}$ on the category $\birs$ obtained. It is characterized as follows:
$\hw_{bir}$ is given by retracts of the birational motives of (smooth)
$U$-schemes; our results also yield a certain description of morphisms 
in the heart. The existence of $w_{bir}$ previously was only known for  $U$
being (the spectrum of) a perfect field (actually, this is the only case in which birational motives were considered previously);
 even in this case we obtain a new
``elementary'' proof of this fact (as well as a new method for the calculation of its heart).
As shown in previous papers of the first author (see Theorem 2.4.2 of \cite{bws}), the existence of a weight structure yields
functorial weight filtrations and weight spectral sequences for any cohomology
theory that factors through birational motives, and a conservative exact
weight complex functor whose target is $K^b(\hw_{bir})$. We also calculate the
Grothendieck group of $\birs$.

Now we list the contents of the paper in more detail.

We introduce some basic notation and definitions in \S\ref{snotata}.

We start \S\ref{srws} with certain calculations in a triangulated category
$\cu$ that contains a {\it weakly negative} class of objects $B$. We prove that
any object in the triangulated subcategory $\du$ of $\cu$ generated by $B$ (i.e., in the
smallest strict triangulated subcategory containing $B$) 
is a cone of a morphism of ``simpler objects'' (i.e., it has a {\it weak
weight decomposition}). It follows that a presentation of a morphism in
$\cu/\du$ as a ``roof'' can be ``simplified'' under certain conditions. This 
technical result easily yields the first of the statements in Theorem \ref{taddlocu}. 
We also recall that $\au[S\ob]_{add}$ coincides with the ``abstract''
localization of $\au$ by $S$ (as constructed by Gabriel and Zisman) if $S$
contains all identities and is closed under direct sums.

In \S\ref{slocadd} we ``compute'' the category $\au[S\ob]_{add}$ very
explicitly. In particular, we prove that in $\au[S^{-1}]_{add}$ any morphism
can be presented as $g\circ s\ob\circ i$, where $i,g\in \mo(\au)$, $s$ is
invertible in $\au[S\ob]_{add}$ (see Proposition \ref{surjww}). This result enables us to finish the proof of Theorem \ref{taddlocu}. Other
significant computations of this section are Propositions \ref{trmatrix} and
(especially) \ref{psum}.
We also compare our results with (some of) the results of the theory of
non-commutative localizations of rings (especially with the ones of \cite{malc}).
Unfortunately, the formulas of this section are rather unpleasant; yet the advantage of our methods over those of the predecessors 
is that we explain the origin of these equalities conceptually.

In \S\ref{sws}  we recall some basics on weight structures. 
For a subcategory $\du$ generated by cones of a set $S$ of $\hw$-morphisms
(where $w$ is a weight structure for $\cu$) we prove the existence of a weight structure on
$\cu/\du$ such that the localization functor is weight-exact. The heart of this
weight structure is the Karoubi-closure of $\hw[S\ob]_{add}$ in $\cu/\du$. It
coincides with the so-called small envelope of $\hw[S\ob]_{add}$ in the case when $w$ is a
bounded weight structure.
We also consider compactly generated triangulated categories; in the case when $\hw$ is the Karoubization of the category of coproducts
of some additive $\au\subset \cu$ (consisting of compact objects), $S\subset \mo(\au)$, and a 
 subcategory $\ddu\subset \cu$ 
generated by $\co(S)$ {\it as a localizing subcategory} we prove: $\cu/\ddu$
possesses a $t$-structure {\it adjacent} to $w_{\cu/\ddu}$ whose heart is the category of
additive functors from $\au[S\ob]^{op}$ to $ \ab$. This is a 
generalization of (some of) 
the results 
of 
 \cite{dwy} and \cite{neeran}.

In \S\ref{sbir} we introduce certain categories of birational motives (via the
method of \cite{kabir}; we are the first to consider birational motives over general base schemes).
It is easily seen that the results of the previous sections can be applied to
them. So, we discuss their ``weights'' and calculate $K_0$-groups of these
categories.

Now, for the convenience of readers already acquainted with weight structures
(in particular with \cite{bws}) we describe the relation of our current results
with those of (\S8.1 of) ibid; we also recall the 
behaviour of $t$-structures in localizations. 

\begin{rrema}\label{rintro}

Several properties of weight structures are quite similar to those of
$t$-structures; in particular, it was proved in \S8.1 of \cite{bws}: if a weight
structure $w$ on $\cu$ restricts to a weight structure on $\du\subset \cu$,
then it also yields a weight structure on $\cu/\du$ (i.e., the localization
functor $l:\cu\to \cu/\du$ is weight-exact). Note here: the ``simplest'' way
to construct a subcategory $\du\subset \cu$ such that $w$ restricts to it is to
generate it 
by an additive subcategory of $\hw$ (moreover, one obtains all
possible $\du$ this way if $w$ is bounded). Now, for the setting of
$t$-structures we certainly have: if 
 $l$ is $t$-exact, then $t_{\cu}$ restricts to a $t$-structure on 
$\du$. Indeed, since $t$-decompositions are canonical, we obtain: both
components of the $t_{\cu}$-decomposition of an object killed by $l$ also
belong to the categorical kernel of $l$.
Now, a crucial distinction of weight decompositions (see Definition
 \ref{dwstr}
 below) from $t$-ones is that they are not canonical; so, this
argument cannot be carried over to weight structures. Moreover, if $\du$ is
generated by cones of any (!) set of $\hw$-morphisms then 
$w$ yields a weight structure on 
$\cu/\du$, though $w$ usually does not restrict to a weight structure on $\du$
(since there are ``not enough'' objects in $\cu_{w=0}\cap \du$, we only have
``weak weight decompositions'' inside $\du$; cf. 
Remark \ref{rcompu}(1) below)! A certain ``explanation'' of this distinction
between weight and $t$-structures is given by the notion of adjacent
structures; see \S\ref{sadja} for more detail.
\end{rrema}

The first author is deeply grateful to prof. L. Barbieri-Viale, the Landau
Network-Centro Volta, and the Cariplo Foundation for the wonderful working
conditions during his staying in the Milano University, where he started writing
this paper. Both authors are very grateful  to prof. D.-C. Cisinski, prof. A.I.
Generalov, prof. B. Kahn, prof. A. Neeman, prof. L.E. Positselski, and the referee for their
interesting comments. 

\section{Some notation and conventions}\label{snotata}

Given a category $C$ and objects $X,Y\in\obj C$, we denote by
$C(X,Y)$ the set of morphisms from $X$ to $Y$ in $C$. We denote by
$\mo(C)$ the class of all morphisms of $C$; $\isom(C)\subset \mo (C)$
is the subclass of all isomorphisms; $\id(\obj C)\subset \isom(C)$ is the
subclass of identity morphisms (of all objects of $C$).

For categories $C',C$ we write $C'\subset C$ if $C'$ is a full 
subcategory of $C$.

Given a category $C$ and objects $X,Y\in\obj C$, we say that $X$ is a {\it
retract} of $Y$ 
 if $\id_X$ can be factored
as $X\stackrel{i}{\to} Y\stackrel{p}{\to}X$ (if $C$ is triangulated or abelian,
then $X$ is a retract of $Y$ if and only if $X$ is its direct summand). We
will call $p$ a {\it retraction}; $i$ will be called a {\it coretraction}.

An additive subcategory $\hu$ of additive category $C$ 
is called {\it Karoubi-closed}
  in $C$ if it
contains all retracts of its objects in $C$. 

The full subcategory $\kar_{C}(\hu)$ of additive category $C$ whose objects
are all retracts of objects of a subcategory $\hu$ (in $C$) will be
called the {\it Karoubi-closure} of $\hu$ in $C$. 

The {\it Karoubization} $\kar(\au)$ (no lower index) of an additive
category $\au$ is the category of ``formal images'' of idempotents in $\au$.
So, its objects are pairs $(A,p)$ for $A\in \obj \au,\ p\in \au(A,A),\ p^2=p$, and the morphisms are given by the formula 
\begin{equation}\label{mthen}
\kar(\au)((X,p),(X',p'))=\{f\in \au(X,X'):\ p'\circ f=f \circ p=f \}.\end{equation}
 The correspondence  $A\mapsto (A,\id_A)$ (for $A\in \obj \au$) fully embeds $\au$ into $\kar(\au)$.
 Besides, $\kar(\au)$ is {\it Karoubian}, i.e., 
 any idempotent morphism yields a direct sum decomposition in 
 $\au$. Equivalently, $\au$ is Karoubian if (and only if) the canonical embedding $\au \to \kar
(\au)$ is an equivalence of categories.
 Recall also that $\kar(\au)$ is
triangulated if $\au$ is (see \cite{ba}).

$\cu$ below will always denote some triangulated category;
usually it will
be endowed with a weight structure $w$. 

A class $D\subset \obj \cu$ will be
called {\it extension-closed}
    if it contains $0$ 
    and for any distinguished triangle $X\to Y\to Z$
in $\cu$ we have: $X,Z\in D\implies
Y\in D$. We will call the smallest extension-closed subclass 
of objects of $\cu$ that  contains a given class $B\subset \obj\cu$ 
  the 
{\it extension-closure} of $B$.

For a $D\subset \obj \cu$ we will denote by $\lan D\ra$ the smallest full Karoubi-closed
triangulated subcategory of $\cu$ containing $D$, whereas the smallest full strict
triangulated subcategory $\du$ of $\cu$ containing $D$ will be called {\it the triangulated subcategory generated by $D$} (i.e., $\lan D\ra$ is the Karoubi-closure of $\du$ in $\cu$).

For $X,Y\in \obj \cu$ we will write $X\perp Y$ if $\cu(X,Y)=\ns$. For
$D,E\subset \obj \cu$ we will write $D\perp E$ if $X\perp Y$ for all $X\in D,\
Y\in E$.
For $D\subset\obj \cu$ we will denote by $D^\perp$ the class
$$\{Y\in \obj \cu:\ X\perp Y\ \forall X\in D\}.$$
Dually, ${}^\perp{}D$ is the class
$\{Y\in \obj \cu:\ Y\perp X\ \forall X\in D\}$.

For an additive category $\au$ we will denote by $K(\au)$ the homotopy category
of (cohomological) complexes over $\au$. Its full subcategory of
bounded complexes will be denoted by $K^b(\au)$. 

For $i, j\in
\z$ with $i\le j$ we will denote by $K^b(\au)^{\ge i}$ (resp. $K^b(\au)^{\le
i}$) the class of those (bounded) complexes that are isomorphic (i.e., homotopy
equivalent) to complexes that have non-zero terms only in degrees $\ge i$
(resp. $\le i$); $K^b(\au)^{[i,j]}$ is the class of complexes isomorphic to
those that  have non-zero terms only in degrees from $i$ to $j$. We will also
use similar notation for the whole $K(\au)$. In our arguments below we usually
can (and will) assume that the corresponding complexes are concentrated in
degrees $\ge i$ (resp. $\le i$, resp. in degrees between $i$ and $j$) themselves.

If $g$ is a morphism in $K(\au)$ then for any $i$ we will use the notation $g^i$
for the corresponding morphism (in $\au$) between degree $i$ components of
complexes and call it the degree $i$ component of $g$. 
If we will say that an arrow (or a sequence of arrows) in $\au$
yields an object of $K^b(\au)$ we will mean by default  that the last object of
this sequence is in degree $0$; yet note  that we will ignore this convention
when the terms of a complex are  indexed by numbers  (in $T^0\to T^1$ the term
$T^1$ is in degree $1$).
    We will always extend a ``finite'' $\au$-complex by $0$'s to $\pm \infty$
    (in order to obtain an object of $K^b(\au)$).
    
For a single morphism $f\in \au(C^{-1}, C^0)$ we will also call the complex 
$C^{-1}\stackrel{f}{\to}C^0$ the cone of $f$.
We will use a similar convention for triangulated categories (i.e., for any
$f\in \cu(X,Y)$ we have a distinguished triangle $X\stackrel{f}{\to}Y\to
\co(f)\to X[1]$).

Given an object $X \oplus Y$ of $\au$ we will use the notation
$in_X$ for the morphism $\begin{pmatrix}
id_X \\
0
\end{pmatrix}:X \to X \oplus Y$ and $pr_X$ for the morphism $\begin{pmatrix}
id_X &
0
\end{pmatrix}:X \oplus Y \to X$ if there is no ambiguity.


\section{Weakly negative classes in triangulated categories, and localizations}\label{srws}

First, we recall certain basic facts on localizations of categories. Everywhere
in the paper except \S\ref{sadja}  we will 
only consider essentially small categories (in order to avoid set-theoretic
difficulties).

By Lemma I.1.2 of \cite{gabr}, for any  category $C$ and any  
set $S\subset \mo C$ there exists an initial object in the category of those
functors from $C$ that converts all elements of $S$ into invertible
morphisms. We will denote the target of this functor by $C[S\ob]$ and call it
the ``abstract'' or the ``Gabriel-Zisman'' localization of $C$ by $S$.  
We have $\obj C[S\ob]=\obj C$; any morphism in $C[S\ob]$ can factorized
into the composition of a chain of morphisms each of which either comes from
$C$ or is inverse to some element of $S$ (in $C[S\ob]$). It follows: any 
morphism in $C[S\ob]$ can be presented as the composition $f_1s_1\ob
f_2s_2\ob\dots f_ns_n\ob$ (a {\it zig-zag}) for some $n\ge 0$, $f_i\in \mo
C$, $s_i\in S$. We will call a single composition $f_is_i\ob$ of this sort a
{\it roof}.

A caution: in 
Theorem \ref{taddlocu} (see also Remark \ref{rnot} below) we 
have already introduced a certain ``additive version'' of
this definition 
for an additive category $\au$; a
priori we only have a comparison functor $Q:\au[S\ob]\to \au[S\ob]_{add}$ that
does not have to be an isomorphism (though it ``usually'' is; see Corollary
\ref{taddloc}). Note that we still have $\obj  \au[S\ob]_{add}=\obj
\au[S\ob]=\obj \au$. 

By the localization of a triangulated category $\cu$ by a triangulated subcategory $\du$ we will mean 
the Verdier quotient $\cu/\du$ (see Remark 2.1.9 of \cite{neebook}). Recall that it naturally equivalent to $\cu/\du'$ for $\du'$ being  the Karoubi-closure  of $\du$ in
$\cu$ (see Remark 2.1.39 of ibid.). 
 Note that $\cu/\du$ is  a
triangulated category (see Theorem 2.1.8 of ibid.); any morphism in it can be presented as a single ``roof''
$fs\ob,$ for some $f, s \in \mo(\cu)$, $\co(s) \in \obj\du$ (see Remark 2.1.13 of ibid.).

\subsection{Weakly negative classes of objects 
and Verdier localizations}

\begin{defi}\label{dwneg}
1. We will say that a set $D$ 
 of objects of $\cu$ is {\it
weakly negative} if $\ns\in D$ and $D\perp (\cup_{i\ge 2} D[i])$. 
 
 2. For any $i,j\in\z$ 
 and a weakly negative $D$ we
  will denote by $D_{[i,j]}$ the extension-closure of $\cup_{i\le l \le j}D[l]$
  if $j\ge i$, and set $D_{[i,j]}=\ns$ if $j<i$. 
 We denote $\cup_{k\le j} D_{[k,j]}$ by $D_{\le j}$; $D_{\ge i}=\cup_{k\ge i}
 D_{[i,k]}$.

 \end{defi}

\begin{rema}\label{rwneg}

1. The weak negativity condition is a weakening of the negativity  one 
that is fulfilled for the hearts of weight structures; see 
Remark  \ref{rstws} below. So, in the case $D=\cu_{w=0}$ for $\cu$ endowed with a weight structure $w$ (see Definition
\ref{dwstr}(IV)) $D_{[i,j]}$ equals 
$\cu_{[i,j]}$.

2. Certainly, the class of objects of the triangulated subcategory of $\cu$ generated by
$D$ (in the sense of \S\ref{snotata}) equals the union of $D_{[-N,N]}$ for all $N>0$.

3. Our basic example of a weakly negative class is $D\subset
K^b(\au)^{[-1,0]}$; its generalization is $D\subset \cu_{[0,1]}$ for $\cu$
endowed with a weight structure $w$. For both of these examples our Proposition
\ref{ploc}(2--4) has consequences 
that are (more or less) easy to formulate. 

\end{rema}

Now we prove a  technical statement that is crucial for this paper.

\begin{pr}\label{ploc}
 Let $\cu$ be a triangulated category and $B\subset \obj \cu$  a weakly negative
 set. Denote by $\du$  the triangulated subcategory of $\cu$ generated by $B$
 (see \S\ref{snotata}; so, $\du\subset \lan B\ra\subset \cu)$). Then the
 following statements are valid.

1. For any 
$m,n\in \z$, 
$M\in B_{[m,n]}$, there exists a distinguished triangle
\begin{equation}\label{ewwd} 
X\to M\to Y\stackrel{f}{\to} X[1]
\end{equation} 
for some  $X\in B_{[m,0]},\  Y\in B_{[1,n]}$ (it may be called a weak weight decomposition of $M$; cf. 
the formula (\ref{wd}) below).

2. For an $n\in \z$ let $X\in {}^{\perp}(B_{\ge 1}),\ Y\in (B_{\le
n-2})^{\perp}$.
Then any $e\in (\cu/\du)(X,Y)$ can be presented as $h \circ q\ob $ for some $Z\in
\obj \cu$, $h\in \cu(Z,Y)$, $q\in \cu(Z,X)$ with $\co (q)\in B_{[n,0]}$.

3. For $n,X,Y,Z, e,h,q$ as above we have: $e=0$ if and only if 
$h$ can be factored in $\cu$ through an element of $B_{[n-1,0]}$.

4. Assume that $X\in {}^{\perp}(B_{\ge 1}),\ Y\in (B_{\le -1})^{\perp}$. Then $\cu(X,
Y)$ surjects onto $(\cu/\du)(X,Y)$.

\end{pr}

\begin{proof}

1. We apply an argument used in the proof of Proposition 3.5.3(8) of \cite{brelmot} and of Theorem 4.3.2(II.1) of \cite{bws}; a similar reasoning can also be found in 
Appendix B of \cite{posat}.

Obviously, we can assume below that $m\le 0,n>0$, since in all the remaining cases the statement is obvious.

We define a certain notion of complexity for elements of  $B_{[m,n]}$. For
an $M\in B[i]$ (for some $i$ between $m$ and $n$) we will say that $M$ has
complexity $\le 0$. If there exists a distinguished triangle $E\to F\to G$, and
$E,G\in B_{[m,n]}$ are of complexity $\le j$ for some $j\ge 0$ (they also could
have smaller complexity) we will say that the complexity of $F$ is $\le j+1$.
 By definition, any element of  $B_{[m,n]}$
 has finite complexity; hence it suffices to verify: for a distinguished
 triangle $O[-1]\to N\to M\to O$ if $N,O\in B_{[m,n]}$ and  possess weak weight
 decompositions (i.e., there exist distinguished triangles $X_N\to N\to Y_N$
 and $X_O\to O\to Y_O$ with $X_N,X_O\in B_{[m,0]},\  Y_N,Y_O\in B_{[1,n]}$), 
 then $M$ possesses a weak weight decomposition also.

Next, we note that the weak negativity of $B$ implies that $B_{[m,0]}[-1]\perp
B_{[1,n]}$; in particular, $X_O[-1]\perp Y_N$.
Hence the morphism $O[-1]\to N$ can be completed to to a commutative square
 $$\begin{CD} X_O[-1] @>{}>> O[-1] \\ @VV{}V @VV{}V \\
X_N @>{}>>N \\
\end{CD} $$
(see the easy Lemma 1.4.1(1) of \cite{bws}). Therefore by 
 Proposition 1.1.11 of \cite{bbd} we can complete the distinguished triangle $N\to M\to O$ to a commutative  diagram
 \begin{equation}\label{dia3na3t}
\begin{CD}
X_N @>{}>>N @>{}>> Y_N \\
@VV{}V@VV{}V @VV{ }V\\ X@>{}>>M
@>{}>>Y\\ @VV{}V@VV{}V @VV{}V\\
X_O@>{}>>O @>{}>> Y_O\\
\end{CD}
\end{equation}
 whose rows and columns are distinguished triangles (for some $X,Y\in \obj \cu$). 
The light and left columns of this diagram yield that  $X\in B_{[m,0]},\  Y\in B_{[1,n]}$. Hence the middle row yields a weak weight decomposition of $M$.

2. 
By the theory of Verdier localizations (see the beginning of this section) $e$ can be presented
as $ h_1 \circ q_1\ob$ for some $q_1:Z_1\to X$ such that $T_1=\co (q_1[-1])\in
\obj \du$.

Applying 
assertion 1 of our proposition to  $T_1[1]$ we  obtain: for some large enough
$N$ there exists  a distinguished triangle $T_2\to T_1\to V\to T_2[1]$
 such that $T_2\in  B_{[-N-1,-1]},\ V\in B_{[0,N-1]}$.

 Since $X\perp V[1]$, we can factorize the morphism $c$ in the distinguished
 triangle $ T_1\to  Z_1 \stackrel{q_1}{\to}  X\stackrel{c}{\to} T_1[1]$ 
  through $T_2[1]$. Hence applying the octahedral axiom of triangulated
  categories 
  we obtain: there exist $Z_2\in \obj \cu$ and a morphism $d:Z_2\to Z_1$
  such that $\co (d)\cong V\in \obj \du$ 
  and $\co(q_1\circ d) \cong T_2[1] \in
  B_{[-N,0]}$.
  We set $q_2=q_1\circ d$, 
  $h_2=h_1\circ d$. We have: $h_2\circ q_2\ob = h_1\circ d \circ d\ob
  \circ q_1\ob = h_1 \circ q_1\ob$ in $\cu/\du$.

Next, applying 
assertion 1 of our proposition to $T_2[2-n]$ we  obtain: there exists  a
distinguished triangle $U\to T_2 
\to T\to U[1]$ 
 such that $U\in B_{[-N-1,n-2]}$ and $T\in B_{[n-1,-1]}$. 
  Hence the octahedral axiom yields: there exist a $Z\in \obj \cu$, an 
  $m\in \cu(Z_2,Z)$ such that $\co (m)=U[1]\in \obj \du$, and a morphism $q\in
  \cu(Z,X)$  with $\co (q)\cong T[1]$.
 Now, since $U\perp Y$, $h_2$ can be factorized as $h\circ m$ for some $h\in
 \cu(Z,Y)$. Thus $ h_2 \circ q_2\ob= h \circ q\ob $ in $\cu/\du$, and we obtain
 the result.
 
 3. Certainly, $e=0$ if and only if $h=0$ in $\cu/\du$. Besides, since any
 $T\in B_{[n-1,0]}$ becomes $0$ in $\cu/\du$, any morphism that factors through
 an object of $B_{[n-1,0]}$ vanishes in $\cu/\du$. It remains to prove the
 converse implication.
 
 We note: our assumptions on $X$ and $\co(q)$ yield that
  $Z\in {}^\perp (B_{\ge 1})$. Indeed, for any $b\in B$, $i\ge 1$, we have a long
  exact sequence $\dots\to \cu(X,b[i])\to \cu(Z,b[i])\to \cu(\co
  (q[-1]),b[i])\to \dots$; so, it remains to apply the weak negativity of $B$.

 Next, a well-known (and easily proven) property of Verdier localizations
 yields: $h$ vanishes in $\cu/\du$ if and only if it factors through an
 object of $\lan B \ra$. Then it certainly also factors through some  $T_1\in
 B_{[-N,N]}$ for some large enough $N$.
 
 Next, we use certain arguments that 
 are quite  similar to the proof of the previous assertion.
 
 Applying 
assertion 1 of our proposition to $T_1$ we  obtain: there exists  a
distinguished triangle $T_2\to T_1 
\to U\to T_2[1]$ 
 such that $T_2\in B_{[-N,0]},\ U\in B_{[1,N]}$. Since  $Z\perp U$, we can
 factorize $h$ into the composition $Z{\to}T_2\stackrel{h'}{\to} Y$.

 Now we apply assertion 1 to $T_2[2-n]$. We obtain a distinguished triangle
 $V\to T_2 
\to T\to V[1]$ for some $T\in B_{[n-1,0]}$, $V\in B_{\le n-2}$. Since $V\perp
Y$, we can factor $h'$ through $T$; hence the same is true for $h$.

 4. It suffices to note: by 
 assertion 2 (applied in the case $n=1$) any $e\in (\cu/\du)(X,Y)$ can be
 presented as $h\circ q\ob$, where the cone of $q$ is zero, i.e., $q$ is an
 isomorphism in $\cu$.

\end{proof}

We will use this proposition several times below. Yet before that we would like
to make some extra ``computational'' remarks. 

\begin{rema}\label{rcompu} 

\begin{enumerate}
\item\label{i1}
 Suppose $w$ is a weight structure on a triangulated category $\cu$ (see
Definition \ref{dwstr} below). Then any set $B\subset \cu_{[0,1]}$ is easily
seen to be weakly negative; the conditions of Proposition \ref{ploc}(2,3) are fulfilled for any
$X\in \cu_{w\le 0}$, $Y\in \cu_{w\ge n}$ (this is immediate from the orthogonality axiom of weight structures, i.e., from Definition \ref{dwstr}(I.iii)).

Besides, in this case we have $B_{[m,n]}\subset \cu_{[m,n+1]}$ for any $m,n\in
\z$. In particular, for the weak weight decomposition (\ref{ewwd}) we have
$X\in \cu_{[m,1]}$, $Y\in \cu_{[0,n]}$ (see Definition \ref{dwstr}(IV) for the definition of the latter classes), whereas for the ``usual'' weight
decomposition (see (\ref{wd})) we would certainly have $X\in \cu_{[m,0]}$.

\item It is easily seen that Proposition \ref{ploc}(4)  cannot be applied (or adjusted) to the
situation when $B$ is an arbitrary subset of $\cu_{[0,2]}$ (or to $B\subset
K^b(\au)^{[-2,0]}$; cf. Remark \ref{rnot} below).  Indeed, for an abelian
$\au$, $\cu=K^b(\au)$, $B$ being the class of complexes coming from all short
exact sequences in $\au$, we have $\cu/\du=D^b(\au)$, whereas 
$D^b(\au)(M,N[r])=\extr^r_{\au}(M,N)$ for any $M,N\in \obj \au$, $r\in \z$;
hence this group can be non-zero for arbitrarily large $r>0$ (for general
$\au,M,N$).
 
One can also obtain an example of this sort using (effective) Voevodsky's
motives (the corresponding calculations are easy in the well understood case
of motives over a perfect field); cf. \S\ref{sbmot} below.

\end{enumerate}

\end{rema}

\subsection{On the universal additive functor inverting a set of morphisms}


\begin{rema}\label{rnot}
1. Till the end of 
\S\ref{scompgen} we will use the following notation:
$\au$ is an additive category,
$\cu=K^b(\au)$, $S$ is some set of $\au$-morphisms, $B\subset K^b(\au)^{[-1,0]}$
is the set of cones of elements of $S$, $D=B_{[0,0]}$ (in the notation of
Definition \ref{dwneg}) is the extension-closure of $B$, $\du\subset \cu$ is
the triangulated category generated by $B$.

2.  We will use the following notation from Theorem \ref{taddlocu} throughout this paper: we  denote the full subcategory of $\cu/\du$ whose objects are those of $\au$  by
$\au[S\ob]_{add}$ 
(note that the theorem justifies this notation). We denote by $u:\au\to \au[S\ob]_{add}$ 
the restriction of the localization functor  $l:\cu\to \cu/\du$ to
$\au$. 

\end{rema}

Now we apply Proposition \ref{ploc} to $\cu$. A more general 
setting will be considered in \S\ref{sws} below.

\begin{pr}\label{paddlocu}


1. For $X,Y\in \obj \au\subset \obj \cu$, any morphism  $e\in (\cu/\du)(X,Y)$ can
be presented as $h \circ q\ob $ for some $Z\in \obj \cu$, $h\in \cu(Z,Y)$, and $q\in
\cu(Z,X)$ such that $\co (q)\in D$.

The morphism $e$ is zero if and only if $h$ can be factored through some $T$
such that there exists a distinguished triangle $C_{-1}\to T\to C_0$ with
$C_i\in D[i]$ (for $i=-1,0$).

2. More generally, for $X\in K^b(\au)^{\ge 0}$, $Y\in K^b(\au)^{\le -n}$ we
have: any   $e\in (\cu/\du)(X,Y)$ can be presented as $h \circ q\ob $ for some $Z\in
\obj \cu$, $h\in \cu(Z,Y)$ such that $q\in \cu(Z,X)$, $\co (q)\in B_{[n,0]}$
(in the notation of Proposition \ref{ploc}).


3. An 
additive functor $F:\au\to \au'$ factors through $u$  (see the Remark above) if and only if $F$ converts all
elements of $S$ into isomorphisms.
\end{pr}
\begin{proof}

1-2. It suffices to note that we can apply Proposition \ref{ploc} to this
setting; note that $B$ is obviously weakly negative.

3. Certainly, $u$ makes all elements of $S$ invertible (since they become
isomorphisms in $\cu/\du$). Conversely, assume that an additive functor $F$ maps
all elements of $S$ into isomorphisms. We consider the exact functor $K^b(F):\cu\to
K^b(\au')$.
In order to verify that $F$ factors through $u$, it suffices to check that
$K^b(F)$ factors through 
 $l$.   The
universal property of 
Verdier localizations yields: 
in order to achieve this, for any distinguished triangle $X\to Y\stackrel{f}{\to} Z\to X[1]$, $X\in \obj
\lan B\ra$, we should prove that $K^b(F)(f)$ is an isomorphism.
Since $K^b(F)$ is exact,  
 it suffices to verify that $K^b(F)(X)=0$ for
such an $X$. Since $K^b(F)$  is additive, we only have to check  that it kills all
objects of $\du$. Applying the exactness of $K^b(F)$ again, we reduce this to
the fact that $K^b(F)$ kills $B$. 

Denote by $G$ the restriction of the corresponding functor from  $\cu/\du\to
K^b(\au')$ on the full subcategory $\au[S^{-1}]_{add}$ of $\cu/\du$.
Since the  image of $G$  is contained in $\au'$ we can 
replace the target of $G$ by $\au'$.
So we obtain a lift 
of $F$ onto the category $\au[S^{-1}]_{add}$. 
\end{proof}

We will finish the proof of Theorem \ref{taddlocu} (i.e., 
we will prove that for any
$F$ 
that converts elements of $S$ into invertible morphism its factorization through $u$ is necessarily unique)
in the next section. Yet 
now we 
recall that $\au[S\ob]$ is ``usually'' additive; hence the theorem mentioned implies that it is isomorphic to
$\au[S\ob]_{add}$.

\begin{coro}\label{taddloc}
Assume that $S$ contains $\id(\obj \au)$ and is closed with respect to the
direct sum operation. Then the natural functor $Q:\au[S\ob]\to \au[S\ob]_{add}$ is an isomorphism of categories.
\end{coro}
\begin{proof}
Theorem \ref{taddlocu} yields: it suffices to verify whether $\au[S\ob]$ has a
natural structure of an additive category that is compatible with the one of
$\au$. Though this was done in \S A3 of \cite{kabiradd} (see also \cite{cismo}),
for the convenience of readers we 
included this statement into the current text as Proposition
\ref{addloc}.
\end{proof}

\begin{pr}\label{addloc}
Suppose $\au$ is an additive category, $S$ is a class of morphisms in $\au$
containing identities and closed under direct sums. Then the localized category
$\au[S^{-1}]$ is additive, as well as  the localization functor $\au\to \au[S^{-1}]$.
\end{pr}
\begin{proof}
To check that the category is additive it suffices to verify that 
(binary)  coproducts and
 products exist in it, $A \times B \cong A \coprod B$ for any pair of objects, 
and every object is a group object.

Firstly, note that every adjoint pair of functors $L : \au \leftrightarrows \hu : R$
induces an adjoint pair $\hat{L} : \au[S^{-1}] \leftrightarrows \hu[T^{-1}] :
\hat{R}$ if $L(S) \subset T$ and $R(T) \subset S$. Indeed, the latter assumptions imply
that the functors $Loc_T \circ L$ and $Loc_S \circ R$ send $S$ and $T$ into
$\isom \hu[T^{-1}]$ and $\isom \au[S^{-1}]$, respectively; hence they induce some functors $\hat{L}$ and $\hat{R}$ on the
localized categories.
Now, the unit and the counit for the  pair $(L,R)$ yield the unit and the counit of the pair
$(\hat{L},\hat{R})$. 

Applying the above observation to the diagonal functor $\au \to \au \times \au$ and its right and left adjoints (i.e., the binary product
and coproduct functor, respectively) we obtain: 
in $\au[S^{-1}]$
all 
binary coproducts and products exist, and $A \times B \cong A \coprod B$ for any objects $A$ and $B$.

It remains to verify that every object of $A[S^{-1}]$ is a group object.
The localization functor $Loc_S$ preserves (binary) products; thus the image of
a group object has the structure of a group object. Since $Loc_S$ is
surjective on objects, we obtain the result.
\end{proof}

\begin{rema}\label{rcismo}
1. For an arbitrary $S\subset \mo \au$ denote by $\sss$ the closure of $S\cup \id(\obj \au)$ in $\mo(\au)$ with respect to the direct sum operation.
Certainly, in $\au[S\ob]_{add}$  
all the elements of $\sss$ become invertible. Hence 
we obtain: $\au[S\ob]_{add}$ $\cong \au[\sss\ob]_{add}\cong \au[\sss\ob]$.

2. It seems that the additivity of $\au[S\ob]$ (for $S$ as in Corollary \ref{taddloc}) is well-known to experts in the field; yet ``classical'' literature does not treat
the additive structure of $\au[S\ob]$ unless $S$ satisfies (either right of left) Ore condition. 

3. Corollary \ref{taddloc} can also be deduced from Theorem \ref{taddlocu} 
 via a ``very explicit'' method. 
For this purpose one should verify that any morphism in $\au[S\ob]$ can be presented as in Proposition \ref{surjww} below; then one should check that two morphisms presented in this form are equal in $\au[S\ob]$ if they become equal in $\au[S\ob]_{add}$ (cf. 
Proposition \ref{psum}(III.2b)).

\end{rema}

\section{On additive localizations: our computations and their comparison with  
the arguments of the other authors}\label{slocadd}

In this section we ``calculate'' the category $\au[S\ob]_{add}$ ``explicitly''.
This allows us to finish the proof of Theorem \ref{taddlocu}. 
We also describe the relation of our results with the theory of
non-commutative localizations of rings.

Some of the formulas of this section are rather unpleasant; most probably this
cannot be avoided.
Yet we note: in contrast to the preceding papers on the subject, we give a
conceptual explanation for the relations established.

\subsection
{Computations for additive localizations}
\label{scompgen}

Let $A,A'$ be some objects of $\obj \au=\obj \au[S\ob]_{add} \subset \obj
\cu/\du$.
Recall from the last section that any morphism between
$A$ and $A'$ can be presented as a certain ``roof'', i.e., as a composition of the
form $h \circ q^{-1}$ for $h,q \in Mor(\cu)$ and $\co(q) \in \du^{[-1,0]}$ (see
Proposition \ref{paddlocu}(2)). We would like to describe morphisms
$\au[S\ob]_{add}$ in terms of the category $\au$. We start with describing
compositions and sums of morphisms in terms of such roofs; next, we reformulate these results 
in terms of $\au$.

\begin{lem}\label{comp}
Let $\phi=f \circ s^{-1}\in (\cu/\du)(A,B)$, $\psi=g \circ t^{-1} \in
(\cu/\du)(B,C)$ be morphisms in $\au[S\ob]$, 
where the domains of $s$ and $t$ are $L, T \in \cu^{[0,1]}$, respectively.

Denote the following complex by $LT$:
 \begin{equation}
  L^0 \oplus T^0  \xrightarrow{\begin{pmatrix}
 d^0_L & 0  \\
 0 & d^0_T \\
 f^0 & -t^0
\end{pmatrix}}
L^1 \oplus T^1 \oplus B
\end{equation}

Then for $q \in \cu(LT, A)$, $r\in \cu(LT,C)$ such that $q^{0} =
\begin{pmatrix} s^{0} & 0
\end{pmatrix}$, $r^{0} = \begin{pmatrix} 0 & g^{0}
\end{pmatrix}$, the composition $\psi \circ \phi$ equals $r \circ q^{-1}$ (in
$\cu/\du$).
\end{lem}
\begin{proof}
Denote by $q'$ an element of $\cu(LT,L)$ such that $q'^{0} = \begin{pmatrix}
id_{L^{0}} & 0 \end{pmatrix},$ $q'^{1} = \begin{pmatrix} id_{L^{1}} & 0 & 0
\end{pmatrix}$. 
Note that $LT$ (together with $q'$) 
is a cone of the composition $L \stackrel{f}\to B \to
\co(t)$ shifted by $[-1]$. Hence $\co(q') = \co(t)$ and $q'$ is invertible in
$\cu/\du$.
The composition $s \circ q'$ equals $q$ since it is determined by its degree
zero component 
and $(f \circ q')^0 = \begin{pmatrix} s^0 & 0 \end{pmatrix}$.

Denote by $r'$ an element of $\cu(LT,T)$ such that $r'^{0} = \begin{pmatrix} 0
& id_{T^{0}} \end{pmatrix},$ $r'^{1} = \begin{pmatrix} 0 & id_{T^{1}} & 0
\end{pmatrix}$.
The composition $g \circ r'$ equals $r$ since it is determined by its degree
zero component and $(r \circ q')^0 = \begin{pmatrix} 0 & g^0 \end{pmatrix}$.
The composition $t \circ r'$ is determined by its degree
zero component and $(t \circ
r')^0 = \begin{pmatrix} 0 & t^0 \end{pmatrix}$. Moreover, the morphism
$h = \begin{pmatrix} 0 & 0 & id_B \end{pmatrix} \in \au(L^1 \oplus T^1
\oplus B, B)$ defines a homotopy between $t \circ r'$ and $f \circ q'$. So, $t
\circ r' = f \circ q'$ in $\cu$. Hence we have 
$$r\circ q^{-1} = g \circ r' \circ q'^{-1} \circ s^{-1} = g \circ t^{-1} \circ
f \circ s^{-1} = \psi \circ \phi$$ in $\cu/\du$.
\end{proof}

Now we are able to describe morphisms in $\au[S\ob]_{add}$ as certain ``short
zig-zags''.

\begin{pr}\label{surjww}
Let $A,A'$ be objects of $\au$. 
Then any $f\in (\cu/\du)(A,A')$  can be presented as the composition $g \circ
s^{-1} \circ i,$ where $A'', T
\in \obj \au$,  $g\in \au(A'',A'), s\in \au(A'',A \oplus T)$, $i\in \au(A,A \oplus T)$ is the canonical coretraction, 
whereas  $\co(s) \in \du.$
\end{pr}
\begin{proof}
Note first:  by  
Proposition \ref{paddlocu}(2), 
$f$ can be presented as the composition $x \circ y^{-1},$ where $x \in
\au(C,A'), y \in \au(C,A)$, $\co (y) \in \du$, 
for some $C = C^{0} \stackrel{d}\to C^{1}\in K^b(\au)^{[0,1]}$.

Now, consider the composition $g \circ s^{-1} \circ i,$ where $i\in \au(A,C^{1}
\oplus A)$ is the canonical coretraction, $s:C^{0} 
\xrightarrow{\begin{pmatrix} d \\ y^{0} \end{pmatrix}} C^{1} \oplus A$, $g = x^{0}\in \au(C^{0}, A').$ We note that $s$ is invertible 
in $\cu/\du$
since its cone is isomorphic to the cone of $y$.

Denote by $p$ the morphism $C \to C^0$ coming from the distinguished triangle
$C^0[-1] \to C^1[-1] \to C \to C^0$. Note that $x = g \circ p$ and $i \circ y =
s \circ p.$

Hence we obtain the result: $g \circ s^{-1} \circ i = g \circ p \circ y^{-1} = x \circ y^{-1} = f.$

\end{proof}

\begin{rema}\label{rdual}
1. One can also generalize the proposition above to the setting of  an arbitrary triangulated category $\cu$ and its subcategory $\du$ generated by a weakly negative class of objects (cf. Proposition  \ref{surj}).

2. In order to finish the proof of Theorem \ref{taddlocu}, we should verify:   for any
$F$ as in  Proposition \ref{paddlocu} its factorization through $u$ is necessarily unique. Consider the smallest (in the sense of inclusions  of 
 morphisms sets) additive subcategory $\au_0$ of $\au[S\ob]_{add}$
 such that $\mo(A_0)$ contains  $u(\mo(\au))$ and also the inverses to all elements of $S$ (note that $\obj \au=\obj \au_0=\obj \au[S\ob]_{add}$). 
It certainly suffices to prove that $\au_0=\au[S\ob]_{add}$.

Hence it suffices to prove: if $s\in \au(A'', A \oplus T)$ and $\co(s)\in D$ then $s$ becomes invertible in 
$\au_0$. 
Since this is equivalent to $\co(s)$ being zero in $K^b(\au_0)$, it suffices to verify: any $L\in D$ is isomorphic to the cone of some $\au$-morphism that becomes invertible in $\au_0$. Since $\au_0$ will not change if we replace $S$ by the closure of $S\cup \id(\obj \au)$ in $\mo(\au)$ with respect to the direct sum operation, for this purpose it suffices to apply Proposition \ref{exts} below.

3. All of the results of this paper are self-dual. In particular, 
Proposition \ref{surjww}  also yields that every $f \in (\cu/\du)
(A,A')$ can be presented as the composition $p \circ s^{-1} \circ h,$ where $p$ is a retraction, $\co (s)\in D$.

\end{rema}

Now we  calculate $D$.

\begin{pr}\label{trmatrix}
The elements of $D$ are exactly those complexes that are isomorphic to 
  cones of morphisms which are given by lower triangular matrices with
elements of  $S$ 
on the diagonal, i.e., 
to complexes of the form
  \begin{equation}\label{etrmor}
  \bigoplus_{j=1}^{n} S^{-1}_{j} \xrightarrow{\begin{pmatrix}
s_{1} & 0 &\dots &0  \\
f_{12}& s_{2} &\dots &0  \\
\\ \dots &\dots &\dots &\dots\\
f_{1n} & f_{2n} & \dots &s_{n}
\end{pmatrix}}
\bigoplus_{j=1}^{n} S^{0}_{j}
\end{equation}
for $n\ge 0$, $s_{1},s_{2},\dots s_{n}\in S$, 
$f_{kl}\in \au( S^{-1}_{k},  S^{0}_{l})$ for all $1\le k<l\le n$.

\end{pr}
\begin{proof} 
Obviously, it suffices to prove: if $t_1,t_2$ are morphisms of the form (\ref{etrmor}), $\co(t_1)\to T \to \co (t_2)\to \co (t_1)[1]$ is a distinguished triangle, then $T$ is isomorphic to the cone of another morphism of the form (\ref{etrmor}). The latter is obvious from the definition of cones (and distinguished triangles) in $K^b(A)$.

\end{proof}

So, the following statement finishes the proof of Theorem \ref{taddlocu}.

\begin{pr}\label{exts}
 Assume that $S$ contains 
 $\id(\obj \au)$ and  is closed under the direct sum operation.
Then 
any morphism of the form described in (\ref{etrmor}) becomes   invertible   in the category $\au[S\ob].$
\end{pr}
\begin{proof}
It suffices to verify that $f$ can be presented as the composition of morphisms each of those is either invertible in $\mo(\au)$
or belongs to $S$.
We prove the latter statement by induction on $n$ (i.e., on the size of the matrix). 
 For $n=1$ there is nothing to prove. Suppose  that the assertion is valid for $n=k\ge 1$.

Now we make the inductive step. 
We present an $f$ of the form (\ref{etrmor}) 
for $n=k+1$ 
as a matrix 
 $s = \begin{pmatrix} t & 0 \\ g & s_n \end{pmatrix},$ where 
 $$t=\bigoplus_{j=1}^{k} S^{-1}_{j} \xrightarrow{\begin{pmatrix}
s_{1} & 0 &\dots &0  \\
f_{12}& s_{2} &\dots & 0  \\
\\ \dots &\dots &\dots &\dots\\
f_{1k} & f_{2k} & \dots &s_{k}
\end{pmatrix}}
\bigoplus_{j=1}^{k} S^{0}_{j},$$ $g=(f_{1n}, f_{2n}, \dots, f_{kn})$. 

Then
  $f = \begin{pmatrix} t & 0 \\ 0 & \id_{S^{0}_{n}} \end{pmatrix}
  \circ \begin{pmatrix} \id_{\bigoplus_{j=1}^{k} S^{-1}_{j}} & 0 \\ f &
  \id_{S^{0}_{n}} \end{pmatrix} \circ \begin{pmatrix} 
  \id_{\bigoplus_{j=1}^{k} S^{-1}_{j}} & 0 \\ 0 & s_n \end{pmatrix}.$
  Applying the inductive assumption to $t$ we decompose 
 it as $e_1\circ e_2\circ\dots\circ e_m$ for some $m\ge 0$, $e_i\in S\cup \isom(\au)$.
  Then $\begin{pmatrix} t & 0 \\ 0 & \id_{S^{0}_{n}} \end{pmatrix}=
  (e_1\bigoplus \id_{S^{0}_{n}})\circ (e_2\bigoplus \id_{S^{0}_{n}})\circ
 \dots\circ (e_m \bigoplus \id_{S^{0}_{n}})$ is a composition of the type desired (note that $S\cup \isom(\au)$ is closed under the operation $- \bigoplus \id_{S^{0}_{n}}$).
Now, the second morphism matrix is invertible in $\au$  since $\begin{pmatrix}
\id_{\bigoplus_{j=1}^{k} S^{-1}_{j}} & 0 \\ -g & \id_{S^{0}_{n}}
\end{pmatrix}$ is its inverse; the third morphism belongs to $S$ itself.

\end{proof}

\begin{rema}\label{rexpl}
This result yields a way to relate the morphism $s$ mentioned in Proposition \ref{surjww} with $S$ very explicitly. Indeed, we can assume that the complex $C$ mentioned in the proof of loc. cit. is an extension in $C^b(\au)$ (i.e., we do not identify homotopy equivalent morphisms of complexes!) of $\co (y)[-1]$ by $A$, whereas $\co (y)$ is a morphism as in (\ref{etrmor}). 
This gives us an explicit description of invertible morphisms that are needed for the decompositions in Proposition \ref{surjww}.

Hence Proposition \ref{psum} can be translated into certain explicit (though quite clumsy) matrix formulas that only mention $\au$ and $S$.
\end{rema}

Now we describe the composition of morphisms in $\au[S\ob]_{add}$, their
addition, and their equality in terms of $\au$ and $D$. 
 Recall here that all morphisms whose cones belong to $D$ become invertible in $\au[S\ob]_{add}$ (see Remark \ref{rdual}(2)); 
one can also describe all $s_i$ in the formulas below using Remark \ref{rexpl}.

By Proposition \ref{surjww}, any 
$\au[S\ob]_{add}$-morphism can be presented as the composition $g \circ s^{-1}
\circ i$. Now we compute 
all the basic categorical operations for morphisms presented in this form (so,
one may say that we describe $\au[S\ob]_{add}$ ``in terms of generators and
relations''). 

\begin{pr}\label{psum}
Let $A_1,A'_1,A_2,A'_2\in \obj \au=\obj \au[S\ob]_{add}$. Present the morphisms 
$f_j\in (\cu/\du)(A_j,A'_j)$ as $f_j=g_j \circ s_j^{-1} \circ i_j$ (for $j=1,2$). 
Here $C_j,C'_j\in \obj \au$, $g_j\in \au(C'_j,A'_j)$,  $s_j\in  \au(C'_j,C_j)$, $\co (s_j)\in D$, $i_j\in \au(A_j,C_j)$ (we do not require $i_j$ to be coretractions). 
Then the following statements are valid.

\begin{enumerate}
\item $f_1 \oplus f_2 = (g_1\oplus g_2) \circ (s_1\oplus s_2)^{-1} \circ
(i_1\oplus i_2)$.

\item Assume that $A_1=A'_2$. Then $f_1 \circ f_2 = \begin{pmatrix} 0 & g_1
\end{pmatrix} \circ \begin{pmatrix} s_2 & 0 \\ -i_1 \circ g_2 & s_1
\end{pmatrix}\ob \circ \begin{pmatrix} i_2 \\ 0 \end{pmatrix}$.

\item Let $A_1=A_2=A$, $A'_1=A'_2=A'$. Then the following statements are 
 valid.

a. $f_1+f_2 = \begin{pmatrix} g_1 & g_2 \end{pmatrix} \circ (s_1\oplus s_2)^{-1}
\circ \begin{pmatrix} i_1 \\ i_2\end{pmatrix}$.

b. Denote the morphism $\begin{pmatrix} i_1 & s_1 & 0 \\ i_2 & 0 & s_2
\end{pmatrix}\in \cu(A \oplus C'_1 \oplus C'_2 , C_1 \oplus C_2)$ by $r$.

Then the following assertions are equivalent:

\begin{itemize}

\item $f_1=f_2$

\item  For $$T = A \oplus C'_1 \oplus C'_2 \xrightarrow{\begin{pmatrix} i_1 &
s_1 & 0 \\ i_2 & 0 & s_2 \end{pmatrix}} C_1 \oplus C_2.$$ there exists a
factorization of the morphism $\begin{pmatrix}0 & g_1 & -g_2 \end{pmatrix}:T
\to A'$ through an object of $\obj\du$ in $C^b(\au)$.

\item  There exist objects $Z,Z',T_1,T_2 \in \obj \au$ and morphisms $k_1\in
\au(Z, T_1),$ $k_2\in \au(T_2,Z')$,  $p \in \au(Z,T_2),g \in \au(T_1,Z'),
\alpha_1^1 \in \au(C_1 \oplus C_2,Z'), \alpha_1^0 \in \au(A \oplus C'_1 \oplus
C'_2,T_1 \oplus T_2)$, and $\alpha_2 \in \au(T_1 \oplus T_2, A')$ such that:
$\co (k_i) \in D$ and the following equality is fulfilled:

$$\begin{pmatrix} \alpha_2 \\ \begin{pmatrix} g & k_2 \end{pmatrix}
\end{pmatrix} \circ \begin{pmatrix} \alpha_1^0 & \begin{pmatrix} k_1 \\ p
\end{pmatrix} \end{pmatrix} = \begin{pmatrix} \begin{pmatrix} 0 & g_1
& -g_2 \end{pmatrix} & 0\\ \alpha_1^1\circ r & 0 \end{pmatrix}.$$

\end{itemize}

\end{enumerate}

\end{pr}
\begin{proof}

\begin{enumerate}
\item Immediate from the functoriality of the direct sum operation (note that the direct sum of invertible morphisms is invertible).

\item By Lemma \ref{comp} $g_1 \circ s_1\ob \circ (i_1 \circ g_2) \circ s_2\ob
= x \circ y\ob$ for $x\in \cu(T,A'_1),$ $y\in \cu(T, C_2)$, $y^0 =
\begin{pmatrix} s_2 & 0\end{pmatrix},$ $x^0 = \begin{pmatrix} 0 & g_1
\end{pmatrix},$ and $T = C'_2 \oplus C'_1 \xrightarrow{\begin{pmatrix} i_1
\circ g_2 & s_1 \end{pmatrix}} C_1 \in K^b(\au)^{[0,1]}.$ According to Lemma \ref{surjww}, we can present $x \circ y^{-1}$ as $g \circ s\ob \circ i,$ where $$i = \begin{pmatrix} id_{C_2} \\ 0 \end{pmatrix}:C_2 \to C_2 \oplus C_1,$$ $$s = \begin{pmatrix} -s_2 & 0 \\ i_1 \circ g_2 & s_1 \end{pmatrix}:C'_2 \oplus C'_1 \to C_2 \oplus C_1,$$ $$g = \begin{pmatrix} 0 & g_1 \end{pmatrix}:C'_2 \oplus C'_1 \to A'_1.$$

So, $f_1 \circ f_2 = g_1 \circ s_1\ob \circ (i_1 \circ g_2) \circ s_2\ob \circ i_2 = 
 g \circ s\ob \circ i \circ i_2 \\
 = g \circ (-id_{C'_2} \oplus id_{C'_1}) \circ (-id_{C'_2} \oplus id_{C'_1})\ob
 \circ s\ob \circ i \circ i_2.$

Writing the above explicitly we get: $$\begin{aligned} f_1 \circ f_2 = g \circ (-id_{C'_2} \oplus id_{C'_1}) \circ (s \circ (-id_{C'_2} \oplus id_{C'_1}))\ob \circ i \circ i_2 \\
= \begin{pmatrix} 0 & g_1 \end{pmatrix} \circ \begin{pmatrix} s_2 & 0 \\ -i_1 \circ g_2 & s_1 \end{pmatrix}\ob \circ \begin{pmatrix} i_2 \\ 0 \end{pmatrix}.\end{aligned}$$

\item
a. 
$f_1 + f_2 = \begin{pmatrix} id_{A'} & id_{A'} \end{pmatrix} \circ f_1 \oplus
f_2 \circ \begin{pmatrix} id_A \\ id_A \end{pmatrix}$. Applying
assertion 1 we present the expression in the form: $\begin{pmatrix} id_{A'}
& id_{A'} \end{pmatrix} \circ g_1 \oplus g_2 \circ (s_1 \oplus s_2)^{-1} \circ i_1 \oplus
i_2 \circ \begin{pmatrix} id_A \\ id_A \end{pmatrix} = \begin{pmatrix}
g_1 & g_2 \end{pmatrix} \circ (s_1\oplus s_2)^{-1} \circ \begin{pmatrix}
i_1 \\ i_2\end{pmatrix}.$

b. By the previous assertion, $f_1-f_2=\begin{pmatrix} g_1 & -g_2 \end{pmatrix}
\circ (s_1\oplus s_2)^{-1} \circ \begin{pmatrix} i_1 \\ i_2\end{pmatrix}$.
Hence by Lemma \ref{comp}, $f_1 - f_2$ equals $x \circ y\ob$ in $\cu/\du$,
where $x = \begin{pmatrix}0 & g_1 & -g_2 \end{pmatrix}\in \cu(T, A'),$ $y =
\begin{pmatrix} id_A & 0 & 0 \end{pmatrix} \in \cu(T, A).$ Applying Proposition
\ref{paddlocu}(1) we obtain: $x \circ y\ob  = 0$ if and only if $x$ factors
through some $T' \in \obj\cu$ that can be presented as a cone of some
$\cu$-morphism between two elements of $D[-1]$.

Firstly, note that the existence of a factorization of $x$ through an object of
 $\du_{[-1,1]}$ follows from (and hence is equivalent to) the existence of
 factorization of any representative of $x$ in $C^b(\au)$ 
 through $\du_{[-1,1]}$. Indeed, the set $H$ of 
$C^b(\au)$-morphisms from $T$ to $A'$ that can be factored through
$\du_{[-1,1]}$ is an additive group (because $\du_{[-1,1]}$ is an additive
subcategory in $C^b(\au)$).
An element $f$ of $C^b(\au)(T,A')$ which is null-homotopic via a homotopy $h$
can be factored through $\co(id_{A'})[-1] \in \du_{[-1,1]}$ via morphisms that are written down in
columns of the following diagram:

\begin{equation}
\begin{CD}
A \oplus C'_1 \oplus C'_2 @>{\begin{pmatrix} i_1 & s_1 & 0 \\ i_2 & 0 &
s_2 \end{pmatrix}}>> C_1 \oplus C_2\\
@VV{f}V @VV{h}V\\ 
A'@>{id_{A'}}>>A'
\\ @VV{id_{A'}}V @. \\
A' @.
\end{CD}
\end{equation}
 
So $H$ also contains all null-homotopic morphisms. 
Thus if some $C^b(\au)$-represen\-tative of $x$ differs from $g$ by a
null-homotopic morphism $h$ (where $g \in H$), then it also belongs to $H$ as
the sum $h + g$, i.e., it can be factored through $\du_{[-1,1]}.$

Let $x' = \begin{pmatrix}0 & g_1 & -g_2 \end{pmatrix}$ be a lift of $x$ to
$C^b(\au)$.
In order to prove the equivalence of the  last two of our assumptions, it remains to 
write down explicitly what does it mean for $x'$ to factor through a cone of a $\cu$-morphism between two elements of $D[-1]$.
A cone of a $\cu$-morphism between two elements of $D[-1]$ has the form $Z
\stackrel{\begin{pmatrix} k_1 \\ p \end{pmatrix}}\to T_1 \oplus T_2
\xrightarrow{\begin{pmatrix} g & k_2 \end{pmatrix}} Z'$ for some $p,g \in
\mo(\au),$ $k_1,k_2 \in S.$

So $f_1=f_2$ if and only if
there exist a factorization of $x'$ through some complex 
$T' = Z
\stackrel{\begin{pmatrix} k_1 \\ p \end{pmatrix}}\to T_1 \oplus T_2
\xrightarrow{\begin{pmatrix} g & k_2 \end{pmatrix}} Z'$,  i.e., morphisms
$\alpha_1 \in C^b(\au)(\co(r), T')$ and $\alpha_2 \in C^b(\au)(T', A')$ such
that $\alpha_2 \circ \alpha_1 = x'.$

We put these morphisms into a diagram whose rows are $T,T',A'$, respectively: 
\begin{equation}\label{redf}
\begin{CD}
@. A \oplus C'_1 \oplus C'_2 @>{\begin{pmatrix} i_1 & s_1 & 0 \\ i_2 & 0 &
s_2 \end{pmatrix}}>> C_1 \oplus C_2\\
@. @VV{\alpha_1^0}V @VV{\alpha_1^1}V\\ 
Z@>{\begin{pmatrix} k_1 \\ p \end{pmatrix}}>>T_1 \oplus T_2
@>{\begin{pmatrix} g & k_2 \end{pmatrix}}>> Z' \\ @VV{}V @VV{\alpha_2
}V @.\\
0 @>{}>> A' @.
\end{CD}
\end{equation}

The diagram yields: $\begin{pmatrix} g & k_2 \end{pmatrix} \circ \alpha_1^0 =
\alpha_1^1 \circ \begin{pmatrix} i_1 & s_1 & 0 \\ i_2 & 0 & s_2
\end{pmatrix}$ and
$\alpha_2 \circ \begin{pmatrix} k_1 \\ p \end{pmatrix} = 0.$

 So, we have $\begin{pmatrix} 0 & g_1 & -g_2 \end{pmatrix} = \alpha_2 \circ
 \alpha_1^0.$

Summarizing all the assumptions made and writing the equalities in a matrix
form we obtain the result.
\end{enumerate}
\end{proof}


\subsection{The comparison of our results with the theory of Cohn's localizations}\label{scompco}

We recall the setting of the ``classical'' theory of non-commutative
localizations of rings (see \cite{cohn}).
The authors apologize for not being able to mention all significant
contributions to this vast subject.

 \subsubsection{On non-commutative localizations (a la Cohn)}\label{sncloc}
 
One considers a set $S$ of matrices over an (associative unital) ring $R$ and
looks for an initial object in the category of ring homomorphisms $R\to R'$
such that all elements of $S$ become invertible over $R'$ (usually only square
matrices were considered; yet non-square ones could also become invertible as
the morphisms between the corresponding free modules).
More generally, instead of matrices one can consider morphisms $P_i\to Q_i$ of
finitely generated projective  $R$-modules and tensor them by $R'$ (see
\cite{scho}). Here we prefer to consider right $R$-modules (and tensor them by
$R'$ from the right).

Denote by $\au$ the category of finitely generated free (resp. projective) 
modules over $R$.
It is easily seen that the problem of finding the ring $R'$ is ``equivalent''
to the calculation of $\au[S\ob]_{add}$. Indeed, denote $R$ considered as a 
module over itself by $\rr$ (so, $R$ acts on it from the left); then any object
of $\au$ is isomorphic to the direct sum of a finite number of copies of $\rr$
(resp.  to a retract of such a direct sum). It follows that the image of $\rr$
in $\au[S\ob]_{add}$ also is a ``generator'',  i.e.,  the objects of
$\au[S\ob]_{add}$  are exactly the direct sums of a finite number of copies
of $\rr$ (resp. are retracts of such direct sums). Hence the 
localization $R'$ of $R$ by $S$ (as mentioned above) is naturally isomorphic to
$\enom_{\au[S\ob]_{add}}\rr$.

For any additive category $\au$ 
consider
the rings $R_{I}=\enom_{\au}(\bigoplus_{A\in I} A)$, where $I$ runs through all
finite subsets of $\obj \au$. If $\au$ is Karoubian, this yields its
presentation as the projective limit of the categories of finitely generated
projective (right) modules over $R_{I}$.
So, it is no surprise that our ``explicit'' description of $\au[S\ob]_{add}$ is
closely related to the (equivalent) descriptions of the 
 localization of $R$ by $S$ given in \cite{gera} and \cite{malc} (whereas the results of the two 
  papers cited are easily seen to be equivalent).

 Moreover, they can probably be deduced from the results of ibid. 
 (via passing to the limit and possibly invoking some methods from
 \cite{scho}; the latter includes considering non-square matrices) in the case of a Karoubian $\au$. The extension to the case of a
 non-Karoubian category could be more difficult.

\subsubsection{The relation of our results with those of Malcolmson}

First, we show that our condition for equality of morphisms in $\au[S\ob]_{add}$
(i.e., Proposition \ref{psum}(III.2b) can be rewritten in a form
similar to the one in \cite{malc} (see the top of page 119).
This means: 
 $f_1=g_1\circ s_1\ob \circ i_1 $ equals $f_2= g_2\circ  s_2\ob \circ  i_2$
(for $g_j,s_j,i_j$ as in 
ibid.) if and only if there exist
$E_1,E_2,R_1,R_2,E\in\obj \au$ and morphisms
$L\in \au(E_1, R_1), M\in \au(E_2, R_2), Q \in \au(C'_1 \oplus C'_2 \oplus E_1
\oplus E_2, E), P \in \au(E, C_1 \oplus C_2 \oplus R_1 \oplus R_2) \in S,$ $u
\in \au(E, A'),v \in \au(A, E),X \in \au(E_1, A'),Y \in \au(A, R_2)$ such that
the following equality is fulfilled:
\begin{equation}\label{emalc}
\begin{pmatrix}
s_1 & 0   & 0 & 0 & i_1 \\
0   & s_2 & 0 & 0 &-i_2 \\
0   & 0   & L & 0 & 0   \\
0   & 0   & 0 & M & Y   \\
g_1 & g_2 & X & 0 & 0
\end{pmatrix} =
\begin{pmatrix}
P \\
u
\end{pmatrix} \circ 
\begin{pmatrix}
Q & v
\end{pmatrix}\end{equation}


One can mimic the
long calculations of 
 \cite{malc} and \cite{scho} (see also \cite{forn}) 
to show that the condition above defines an equivalence relation on the sets of
triples $(g,s,i),$ for $g,s,i$ as in Proposition \ref{psum}. 
 Moreover,
 we can define the additive category $\au'$ with $\obj \au' = \obj \au$ and
 morphisms defined by the triples mentioned modulo this equivalence relation,
 their compositions and addition defined via the relations of parts I-III.1 of ibid.
Hence
there is a functor $M:\au \to \au'$ sending $f \in \mo(\au)$ to $(f, \id, \id)$;
it sends elements of $S$ into isomorphisms. So, there is a functor $\au[S\ob]
\to \au'$.


Now we construct the inverse functor $I:\au' \to \au[S\ob].$ We define it  on
objects in the obvious way.
To define it on morphisms consider the map from $\mo \au'$ to $\mo
\au[S\ob]_{add}$ sending $(g,s,i)$ to $g \circ s\ob \circ i$.
Note that it is well-defined. Indeed, let
if $(g_1, s_1, i_1)$ and $(g_2, s_2, i_2)$ be equivalent  with respect to
Malcolmson's relation, i.e., suppose 
we have  $L,M,Q,P,u,v,x,y$ as in
Malcolmson's formula.
In order to prove that $g_1 \circ s_1\ob \circ i_1=g_2 \circ s_2\ob \circ i_2$,
we 
 apply Proposition \ref{psum}(III.2b). Consider the factorization of
 $\begin{pmatrix} 0 & g_1 & -g_2\end{pmatrix}$ 
through the complex $T'$:
\begin{equation}
\begin{CD}
@. A \oplus C'_1 \oplus C'_2 @>{\begin{pmatrix} i_1 & s_1 & 0 \\ i_2 & 0 &
s_2 \end{pmatrix}}>> C_1 \oplus C_2\\
@. @VV{\begin{pmatrix} v & Q \circ \begin{pmatrix} in_{C'_1} &
-in_{C'_2}\end{pmatrix} \end{pmatrix}}V @VV{\begin{pmatrix} in_{C_1}
& -in_{C_2} \end{pmatrix}}V\\
E_2@>{Q \circ in_{E_2}}>>E @>{pr_{C_1 \oplus C_2 \oplus R_1} \circ P}>> C_1
\oplus C_2 \oplus R_1 \\ @VV{}V @VV{u
}V @.\\
0 @>{}>> A' @.
\end{CD}
\end{equation}
(see the end of \S\ref{snotata} for the notation).

Note that $T'$
 belongs to $\obj \du,$ 
 since in  $\cu/\du$ it is isomorphic to the contractible complex $E_2
 \stackrel{in_{E_2}}\to C'_1 \oplus C'_2 \oplus E_1 \oplus E_2
 \stackrel{pr_{C'_1 \oplus C'_2 \oplus E_1}}\to C'_1 \oplus C'_2 \oplus E_1$
 via the obvious isomorphism
\begin{equation}
\begin{CD}
E_2 @>{in_{E_2}}>> C'_1 \oplus C'_2 \oplus E_1 \oplus E_2 @>{pr_{C'_1 \oplus
C'_2 \oplus E_1}}>> C'_1 \oplus C'_2 \oplus E_1\\
@VV{id_{E_2}}V @VV{Q}V @VV{s_1 \oplus s_2 \oplus L}V\\ 
E_2@>{Q \circ in_{E_2}}>>E @>{pr_{C'_1 \oplus C'_2 \oplus R_1} \circ P}>> C_1
\oplus C_2 \oplus R_1
\end{CD}
\end{equation}

Our map on morphisms preserves composition and addition since our composition
and addition are obviously compatible with Malcolmson's ones.
So $I$ is an additive functor such that $I \circ M = Q$, where $Q:\au \to
\au[S\ob]_{add}$ is the localization functor. It is unique since it is defined
on products of morphisms coming from $\au$. So, $\au'$ coincides with
$\au[S\ob]_{add}$; hence Malcolmson's formula can be used for comparing
morphisms in the localized category indeed.


\begin{rema}\label{cnc}

1. 
We can modify the construction of localization given in Proposition \ref{psum} as follows.
We may assume that every triple in loc. cit. has the form $(f,s,i),$
where $f,i \in \mo \au$ and $s \in \hat S$ (instead of $\co(s) \in D$), where
$\hat S$ is the class of lower triangular matrices with elements of $S$ and
identity morphisms on the diagonal.

Indeed, it suffices to verify that every triple $(f,s,i)$ with $\co(s) \in D$ is
equivalent to a triple $(f',s',i')$ with $s' \in \hat S.$

By Proposition \ref{trmatrix}, if $\co(s) \in D$ then there exists a homotopy
equivalence from $\co(s')$ to $\co(s),$ where $s' \in \hat S$.
Now, in the proof of Proposition \ref{surjww} we may take $T=\co(A \to
\co(s'))[-1],$ where $s' \in \hat S.$ So, $s = \begin{pmatrix}d^0 \\
y^0\end{pmatrix} \in \au(T^0, T^1 \oplus A)$ belongs to $\hat S$ because
its cone is a cone of morphism of complexes $\co(id_A)[-1] \to \co(s'),$ i.e.,
$s$ is a lower 
triangular matrix with $id_A$ and $s'$ on the diagonal.

2. Now we show that our results imply the main statements of \cite{malc}.

As explained in the 
section \ref{sncloc}, Proposition \ref{psum} 
can be used  to obtain a complete 
 description of non-commutative localizations of rings. Moreover, 
 as we have shown above, 
 we can use the condition (\ref{emalc}) instead of the last of equivalent
 conditions in part III.2b of ibid. Now we check that this combination yields
 exactly the description of non-commutative localizations given in \cite{malc}.

We have: 
the non-commutative localization of a ring $R$ with respect to a set of 
matrices $S$ over $R$ is naturally isomorphic to $R' =
\enom_{\fgf(R)[S\ob]_{add}}(\rr)$, where $\fgf(R)$ is the category of finitely
generated right free modules over $R$.
Denote by $\hat S$ the set of lower triangular 
matrices with elements of $S$
on diagonal.
Proposition \ref{psum} and part 1 of this remark yield: $R'$ can be
described as the set of triples $(g,s,i)$ modulo the equivalence relation
(\ref{emalc}), where $s \in \hat S$ 
 is of size $n \times m$ for some $n,m>0$, $i \in \fgf(R)(\rr,\rr^m)$ and $g
 \in \fgf(R)(\rr^n,\rr)$, i.e., $i$ is a column and $g$ is a row over $R$ of size
 $n$. The operations on this set are defined using Proposition \ref{psum}.

Thus one can easily see: we obtain exactly the construction of the
non-commutative localization of $R$ described in \cite{malc}.

Certainly, one can also use the last of equivalent conditions from Proposition \ref{psum}(III.2b) instead of
(\ref{emalc}) here.

\end{rema}

A major advantage of our methods is that we define $\au[S\ob]_{add}$ as a full
subcategory of $K^b(\au)/\du$. So, we do not have to verify that our
presentation of morphisms and the relations on them (see Proposition
\ref{psum}) does yield an additive category indeed (in contrast with the
arguments of \cite{gera} and \cite{malc}). Besides, the consideration of
$K^b(\au)/\du$ explains the origin of the relations obtained (their analogues
in the papers cited look quite ad hoc). In particular, the {\it
multiplicativity} condition for $M$  in \cite{malc} (and also in \cite{gera})
corresponds to the extension-closedness of our $D$ (cf. (\ref{etrmor})).

Also, ibid. contains no information about how (\ref{emalc}) was obtained or
how it can be understood (in contrast to Proposition \ref{psum}(III.2)).
Besides,
though both our and Malcolmson's conditions for the equality of morphisms in
$\au[S\ob]_{add}$
are hard to verify (in general), there are some advantages of using our one. 
Our formula uses only 6 free variables, and Malcolmson's formula uses 7 ones. %

Lastly, one can easily note that Theorem \ref{taddlocu} 
yields a mighty tool
for computations in $\au[S\ob]_{add}$ (we thoroughly demonstrated this in the
previous section; cf. Proposition \ref{paddlocu}).


We will relate our methods and results with the ``alternative triangulated
approach to additive localizations'' of \cite{neeran} and \cite{dwy} in
\S\ref{sadja} below.
Here we only note that non-commutative localizations  seem to be really nicely
related to weight structures (that we will treat in the next section).
Certainly, one of the reasons to say this is Theorem \ref{twloc} below. We
would also like to say that though Proposition \ref{ploc} (that is crucial for
this paper) was not formulated in the terms of weight structures, it is
certainly closely related to them. In particular, part 1 of the Proposition
is a modification of the axiom \ref{dwstr}(I.iv), whereas 
the formulations of parts 2--4 are motivated by Remark 
\ref{rcompu}(1).



\section{On weight structures in localizations}\label{sws}

We start this section by recalling the definition and basic properties of
weight structures. 
They 
enable us to generalize the results of the previous sections and prove for
$S\subset \mo \hw$: the localized category $\cu/\du$ is endowed with a weight
structure whose heart is the Karoubi-closure of $\hw[S\ob]_{add}$ in $\cu/\du$.
The proofs are minor modifications of the arguments above (that correspond to
the case $\cu=K^b(\au)$, $w$ is the 'stupid weight structure'). 
Next, we  recall the notion of adjacent structures; we prove (in the case of a
certain compactly generated $\ddu\subset \cu$) that $\cu/\ddu$ possesses a
$t$-structure {\it adjacent} to the corresponding $w_{\cu/\ddu}$, and calculate
its heart.

\subsection{A reminder on weight structures}

\begin{defi}\label{dwstr}

I A pair of subclasses $\cu_{w\le 0},\cu_{w\ge 0}\subset\obj \cu$ 
will be said to define a weight
structure $w$ for a triangulated category  $\cu$ if 
they  satisfy the following conditions:

(i) $\cu_{w\ge 0},\cu_{w\le 0}$ are 
Karoubi-closed in $\cu$
(i.e., contain all $\cu$-retracts of their objects).

(ii) {\bf Semi-invariance with respect to translations.}

$\cu_{w\le 0}\subset \cu_{w\le 0}[1]$, $\cu_{w\ge 0}[1]\subset
\cu_{w\ge 0}$.

(iii) {\bf Orthogonality.}

$\cu_{w\le 0}\perp \cu_{w\ge 0}[1]$.

(iv) {\bf Weight decompositions}.

 For any $M\in\obj \cu$ there
exists a distinguished triangle
\begin{equation}\label{wd}
X\to M\to Y
{\to} X[1]
\end{equation} 
such that $X\in \cu_{w\le 0},\  Y\in \cu_{w\ge 0}[1]$.

II The category $\hw\subset \cu$ whose objects are
$\cu_{w=0}=\cu_{w\ge 0}\cap \cu_{w\le 0}$ and morphisms are $\hw(Z,T)=\cu(Z,T)$ for
$Z,T\in \cu_{w=0}$,
 will be called the {\it heart} of 
$w$.


III $\cu_{w\ge i}$ (resp. $\cu_{w\le i}$, resp.
$\cu_{w= i}$) will denote $\cu_{w\ge
0}[i]$ (resp. $\cu_{w\le 0}[i]$, resp. $\cu_{w= 0}[i]$).

IV We denote $\cu_{w\ge i}\cap \cu_{w\le j}$
by $\cu_{[i,j]}$ (so it equals $\ns$ for $i>j$).

$\cu^b\subset \cu$ will be the category whose object class is $\cup_{i,j\in \z}\cu_{[i,j]}$.

V We will  say that $(\cu,w)$ is {\it  bounded}  if $\cu_b=\cu$ (i.e.,
$\cup_{i\in \z} \cu_{w\le i}=\obj \cu=\cup_{i\in \z} \cu_{w\ge i}$)s.

VI Let $\cu$ and $\cu'$ 
be triangulated categories endowed with
weight structures $w$ and
 $w'$, respectively; let $F:\cu\to \cu'$ be an exact functor.

$F$ will be called {\it left weight-exact} 
(with respect to $w,w'$) if it maps
$\cu_{w\le 0}$ into $\cu'_{w'\le 0}$; it will be called {\it right weight-exact} if it
maps $\cu_{w\ge 0}$ into $\cu'_{w'\ge 0}$. $F$ is called {\it weight-exact}
if it is both left 
and right weight-exact.

VII Let $\hu$ be a 
full subcategory of a triangulated category $\cu$.

We will say that $\hu$ is {\it negative} if
 $\obj \hu\perp (\cup_{i>0}\obj (\hu[i]))$.

VIII The {\it small envelope} of an additive category $\au$
is the  category $\au'\subset \kar(\au)$ whose objects are $(X,p)$ for $X\in\obj \au$ and
$p\in \au(X,X)$ such that  $ p^2=p$ and there exist $Y\in \obj \au$ and
$q\in \au(X,Y)$, $s\in \au(Y,X)$ satisfying $sq=1-p$, $qs=\id_Y$. So, the  morphism groups in $\au'$ are given by the  
formula (\ref{mthen}). 

\end{defi}

\begin{rema}\label{rstws}

1. A  simple (and yet very useful for us) example of a weight structure comes from the stupid
filtration on 
$K^b(\au)$ (or for $K(\au)$) for an arbitrary additive category
 $\au$. 
In this case
$K^b(\au)_{w\le 0}$ (resp. $K^b(\au)_{w\ge 0}$) will be the class of complexes that are
homotopy equivalent to complexes
 concentrated in degrees $\ge 0$ (resp. $\le 0$).  The heart of this weight structure 
is the Karoubi-closure  of $\au$
 in 
 $K^b(\au)$. Note that we have $K^b(\au)^{[-j,-i]}\subset K^b(\au)_{[i,j]}$ for any $i,j\in \z$; we have an equality if $\au$ is Karoubian, but in general $K^b(\au)_{[0,0]}$ is the small envelope of $\au$ (so it is not necessarily equivalent to $\au$; cf. Proposition \ref{pbw}(\ref{igen})).

3. 
A weight decomposition (of any $M\in \obj\cu$) is (almost) never canonical.
 
4. In the current paper we use the ``homological convention'' for weight structures; 
it was previously used in \cite{hebpo}, \cite{wild},   and  \cite{btrans}, whereas in 
\cite{bws} and in \cite{bger} the ``cohomological convention'' was used. In the latter convention 
the roles of $\cu_{w\le 0}$ and $\cu_{w\ge 0}$ are interchanged, i.e., one
considers   $\cu^{w\le 0}=\cu_{w\ge 0}$ and $\cu^{w\ge 0}=\cu_{w\le 0}$. So,  a
complex $X\in \obj K(\au)$ whose only non-zero term is the fifth one (i.e.,
$X^5\neq 0$) has weight $-5$ in the homological convention, and has weight $5$
in the cohomological convention. Thus the conventions differ by ``signs of
weights''; 
 $K(\au)_{[i,j]}$ is the class of retracts of complexes concentrated in degrees
 $[-j,-i]$. 
 
 We also recall: in \cite{konk}
D. Pauksztello
introduced weight structures independently (and called them
co-t-structures). 

 5. The orthogonality axiom in Definition \ref{dwstr}(I) immediately yields that $\hw$ is negative in $\cu$.
 We will invoke a certain converse to this statement 
 below.

\end{rema}

Now we recall those properties of weight structures that
will be needed below (and can be easily formulated).

\begin{pr} \label{pbw}
Let $\cu$ be a triangulated category; we will assume 
that $w$ is a fixed 
weight structure on it  everywhere except assertion  \ref{igen}.

\begin{enumerate}

\item\label{idual} 
For $C_1=C_{w\le 0}$ and $C_2=C_{w\ge 0}$, the classes
$(C_2^{op}, C_1^{op})$ define a weight structure on $\cu^{op}$.

\item\label{iext} 
 $\cu_{w\le 0}$, $\cu_{w\ge 0}$, and $\cu_{w=0}$
are extension-stable. 

 \item\label{iwd0} For any weight decomposition of $M\in \cu_{w\ge 0}$ (see
(\ref{wd})) we have $X\in \cu_{w=0}$.

\item\label{ibougen}
The category $\cu^b$ is a triangulated subcategory of $\cu$;
$w$ restricts to a bounded weight structure on $\cu^b$ (i.e., we consider the classes $\cu_{w\le 0}\cap \obj\cu^b$ and $\cu_{w\ge 0}\cap \obj\cu^b$) 
whose heart equals $\hw$.

\item\label{igenw0}
If  $w$  is bounded, then  $\cu$ is generated by $\cu_{w=0}$ (see \S\ref{snotata}). 

\item\label{i01}  $\cu_{[0,1]}$ consists exactly of cones of morphisms in $\hw$.

\item \label{igen}
Assume that  $\hu\subset  \cu$ is negative, additive, and generates $\cu$. 
Then
there exists a unique weight structure $w$ on $\cu$  such that $\hu\subset \hw$.
It is bounded; its heart is equivalent to the small envelope of $\hu$. 

Moreover, $\cu_{w\le 0}$ is the
   smallest Karoubi-closed extension-stable subclass of $\obj \cu$ containing 
   $\cup_{i\le 0}\obj \hu[i]$;
$\cu_{w\ge 0}$ is the
   smallest Karoubi-closed extension-stable subclass of $\obj \cu$ containing 
   $\cup_{i\ge 0} \obj \hu[i]$.

\item\label{idemp} Assume that $\cu$ is endowed with a bounded weight structure
$w$. Then there exists a unique weight structure $w'$ on the Karoubization
$\cu'$ of $\cu$ such that the embedding $\cu\to \cu'$ is weight-exact. It is
bounded; its heart is the Karoubization of $\hw$.

\item\label{iwwc} 
A certain additive {\it weak weight complex} functor
$t:\cu\to K_\w (\hw)$ is defined, where $K_\w(\hw)$ is a certain {\it weak
homotopy category of complexes}. $t$ and $K_\w(-)$ satisfy the following
properties.

(i) For any additive category $\au$  we have: there is a natural conservative additive
functor $p=p_{\au}:K(\au)\to K_\w(\au)$ that is bijective on objects and
surjective on morphisms. 
It commutes with $[i]$ for any $i\in\z$; we
have $K(\au)(X,Y)\cong K_w(\au)(p(X),p(Y))$ for any $X\in K(\au)^{\le 0}$, $Y\in
K(\au)^{\ge 0}$.

(ii) There also exists a certain bounded analogue $K^b_\w(-)\subset K_\w(-)$.
If $w$ is bounded then $t$ is conservative and  can be factored through
$K^b_\w(\hw)$.

(iii)  If $X[-1]\to Y\stackrel{f}{\to} Z\to X$ is a distinguished triangle in
$\cu$ then there exists a  lift of $t(f)$ to a $t'(f)\in K(\au)(t(Y),t(Z))$
such that  $t(X)\cong \co(t'(f))$.

(iv) Let $S=\{S^{-1}_i\stackrel{s_i}\to S^0_i,\ i\in I\}\subset \mo \hw$; let
$B\subset \cu_{[0,1]}$ be the set of cones of elements of $S$ (see assertion
\ref{i01}), $D$ is the extension-closure of $B$. Then $D\subset \cu_{[0,1]}$;
$t(D)$ consists of cones of morphisms given by  lower triangular matrices with
elements of $S$ 
on the diagonal; see (\ref{etrmor}).

\item\label{iwe} Assume that $w$ be bounded and that $w'$ is a weight structure
for a triangulated category $\cu'$. Then an exact functor $F:\cu\to \cu'$ is
weight-exact if and only if $F(\cu_{w=0})\subset \cu_{w'=0}$.

\end{enumerate}
\end{pr}
\begin{proof}


All of these statements except the two last ones 
 can be found in \cite{bws} (pay attention to Remark
\ref{rstws}(4)!); see Remark 1.1.2(1), Proposition 1.3.3(3,6),  Proposition
1.3.6(1,2), Corollary 1.5.7, Proposition 1.5.6(2),  Theorem 4.3.2(II),  and
Proposition 5.2.2 of ibid., respectively.

The remaining assertions also easily follow from the results of ibid. 
 Part (i) of assertion \ref{iwwc} is immediate from the definition of $K_w(-)$
 (this is Definition 3.1.6 of ibid.).
Part (ii) follows from Theorem 3.3.1(I,V) of ibid. 
Part (iii) 
 is also an easy consequence of the {\it weak
exactness} of $t$ (see loc. cit.);
it implies part (iv) immediately (cf. Proposition \ref{trmatrix}).

Lastly, assertion \ref{iwe} is an easy consequence of assertion \ref{igen}.

\end{proof}

\subsection{The relation of weight structures with ``additive localizations''}

In this subsection we assume that $w$ is a weight structure on a triangulated
category $\cu$ and $\du \subset \cu$ is a full triangulated subcategory
generated (see \S\ref{snotata}) by 
 a set of objects $B \subset \cu_{[0,1]}$. 

We will need a generalization of Proposition \ref{surjww} to this setting. 
Denote by $\hu$ the full subcategory of $\cu/\du$ whose objects are
$\cu_{w=0}$.

\begin{pr}\label{surj}
Let $A,A'$ be objects of $\hu$.

Then any $f\in (\cu/\du)(A,A')$ can be presented as the composition $g \circ
s^{-1} \circ i,$ where $A'', T \in \obj
\hu$, $g\in \hu(A'', A')$, $ s\in \hu(A'', A \oplus T)$,  $i:A \to A \oplus T$ is the canonical coretraction, 
whereas 
 $\co(s) \in \obj \du.$

\end{pr}
\begin{proof}

By  Proposition 
\ref{paddlocu}(1), 
$f$ can be presented as the composition $x \circ y^{-1}$ for some $C \in \cu_{[0,1]}$, $x \in
\cu(C,A'), y \in \cu(C,A)$, such that $\co (y) \in \obj\du$, 
Choose a weight decomposition of $C$ (in $\cu$); we obtain a distinguished triangle
$C^1[-1] \stackrel{d'}\to C \stackrel{p}\to C^{0} \stackrel{d}\to C^{1},$ where
$C^{0}, C^{1} \in \obj \hu.$ 
Since $y \circ d' = 0$, 
$y$ factors through $p$. So,
there exists a morphism $y':C^0 \to A$ such that $y' \circ p = y$ Similarly,
there is a morphism $x':C^0 \to A'$ such that $x' \circ p = x$.

Now consider the composition $g \circ s^{-1} \circ i,$ where $i:A \to C^{1}
\oplus A$ is the canonical coretraction, $s:C^{0} \stackrel{\begin{pmatrix} d
\\ y' \end{pmatrix}}\to C^{1} \oplus A$, $g = x':C^{0} \to A'.$ Note that $s$
is  invertible in $\cu/\du$ indeed:
the octahedron diagram for the commutative triangle given by the equality
$\begin{pmatrix} id_{C^1} & 0 \end{pmatrix} \circ s = d$ yields a distinguished
triangle $C \stackrel{y}\to A \to \co (s)$; thus $\co(s)$ is isomorphic to
$\co(y).$

Next,  there is an equality $i \circ y = s \circ p:$
$$i \circ y = \begin{pmatrix} 0 \\ id_{A} \end{pmatrix} \circ y =
\begin{pmatrix} 0 \\ y \end{pmatrix} = \begin{pmatrix} d \circ p \\ y' \circ p
\end{pmatrix} = \begin{pmatrix} d \\ y' \end{pmatrix} \circ p = s \circ p.$$

Hence $g \circ s^{-1} \circ i = g \circ p \circ y^{-1} = x \circ y^{-1} = f$.

\end{proof}

Now we 
prove the main result of this section.

\begin{theo}\label{twloc}

The following statements are valid.  
 
 1. The Karoubi-closures of the sets $(\cu_{w\le 0},\cu_{w\ge 0})$ in $\cu/\du$
 yield a weight structure $w_{\cu/\du}$ for $\cu/\du$ (such that the
 localization functor is weight-exact).

2. 
$\hw_{\cu/\du}$ is the Karoubi-closure  of $\hu$ in $\cu/\du$.

3. If $w$ is bounded, then $w_{\cu/\du}$ also is; the heart of $w_{\cu/\du}$ is isomorphic to the small envelope of $\hu$.

4. Let $S$ denote 
some set of $\hw$-morphisms 
such that $B$ consists 
 of cones of elements of $S$ (see Proposition \ref{pbw}(\ref{i01})). 
Then $\hu$ is canonically isomorphic to  $\hw[S\ob]_{add}$. 

5. More generally, for an additive category $\au\subset \hw$ assume that $S\subset \mo(\au)$,
$\hw=\kar_{\cu}\au$, and $B=\co(S)$. Then there exists a unique weight structure $w_{\cu/\du}$
for ${\cu/\du}$ such that the localization functor is weight-exact and yields
an isomorphism of $\au[S\ob]_{add}$ with the full subcategory  $\hu'\subset
\hw_{\cu/\du}:\ \obj \hu'=\au$.

\end{theo}
\begin{proof}
1. We should verify that the Karoubi-closures of  $(\cu_{w\ge 0},\cu_{w\le 0})$ satisfy all the axioms of  weight structures. 

These Karoubi-closures  are obviously semi-invariant with respect to translations. 
The orthogonality axiom follows from Proposition \ref{ploc}(4) easily (cf. Remark \ref{rcompu}(1)); it suffices to note that $\cu_{w\le 0}{\perp}\cu_{w\ge 1}\supset B_{\ge 1}$, and $   B_{\le -1}\subset \cu_{\le 0} {\perp} \cu_{w\ge 1}$ (in $\cu$; by Definition \ref{dwstr}(I.iii)).
Lastly, any object of $\cu/\du$ possesses a weight decomposition that comes from $\cu$. 

2. 
$\cu_{w=0}\subset (\cu/\du)_{w_{\cu/\du}=0}$ by the definition of $w_{\cu/\du}$.

Now we prove that any object of $(\cu/\du)_{w_{\cu/\du}=0}$ is a retract of an
object of $\cu_{w=0}$ (in $\cu/\du$). Actually, this fact is valid for any weight-exact functor that is surjective on objects. 

We repeat an argument used in \S8.1 of \cite{bws}.
Let $Z$ belong to $(\cu/\du)_{w_{\cu/\du}=0}\subset \obj \cu$. 
Shifting a weight decomposition of $Z[1]$ in $\cu$ by $[-1]$ we obtain a distinguished triangle 
$T\to Z\to U$ with $T\in \cu_{w\le -1}$, $U\in \cu_{w\ge 0}$.
In $\cu/\du$ we have $T\in (\cu/\du)_{w_{\cu/\du}\le -1}$; hence
$T\perp Z$ and we obtain $U\cong_{\cu/\du} Z\bigoplus T[1]$. Therefore, in
$\cu/\du$ the object $Z$ is a  retract of $U$, whereas $U\in
((\cu/\du)_{w_{\cu/\du}=0})\cap \cu_{w\le 0}$.

 Now, applying the dual
 argument to $U$ (see Proposition \ref{pbw}(\ref{idual})), 
 we obtain: 
$U$ in $\cu/\du$ is a retract of
 some $W\in \obj \cu_{w=0}$ (here  $W\to U\to Y$ is a weight decomposition in $\cu$; we apply part \ref{iwd0} of loc. cit.).

3. Since the localization functor is weight-exact, $w_{\cu/\du}$ is bounded. Next, we have $\hu\subset \hw$; hence $\hu$ is negative (see Remark \ref{rstws}(5)). Proposition \ref{pbw}(\ref{iext}) yields that $w_{\cu/\du}$ is the weight structure corresponding to $\hu$ via assertion \ref{igen} of loc. cit. Hence the assertion mentioned yields the result.

4. We follow the scheme of the proof of Theorem \ref{taddlocu}. 
First, we
prove that any additive functor $F:\hw\to \au'$ that converts all elements of $S$
into invertible morphisms factors through the functor $u:\hw\to \hu$ (actually,
it suffices to verify this for $F$ being the localization functor $\hw\to\hw[S\ob]_{add}$).

We consider the composition $G=K_\w(F)\circ t:\cu\to K_\w(\hw[S\ob]_{add})$.
Obviously, it suffices to verify that $G$ factors through the localization
functor $\cu\to \cu/\du$. 
By the universal property of ``abstract'' localizations, it suffices to verify
for any distinguished triangle $X\to Y\stackrel{f}{\to} Z\to X[1]$ that $G(f)$
is an isomorphism if $X\in \obj \lan B\ra$.
By Proposition \ref{pbw}(\ref{iwwc}(iii)) 
 it suffices to verify
that $G(X)=0$ for such an $X$. Since $G$ is additive, it suffices to verify
that it kills all objects of $\du$. Applying loc. cit. again (repeatedly) we
reduce this to the fact that it kills all elements of $B$. The latter is
obvious since all elements of $B$ die in $K(\hw[S\ob]_{add})$,
 whereas morphism groups in  $K(\hw[S\ob]_{add})$ surject onto those in
 $K_\w(\hw[S\ob]_{add})$.

By Theorem \ref{taddlocu}, there exists a unique additive functor
$R:\hw[S\ob]_{add}\to \hu$ compatible with $u$ that is bijective on objects.
Since there also is a functor in the inverse direction (as we have just
proved), $R$ is injective on morphisms.

As is the proof of 
loc. cit., it remains to verify that there are no ``redundant'' morphisms in $\hu$. 
The latter fact  is immediate from Proposition \ref{surj}: every morphism in
$\hu$ can be represented as the composition $g \circ s^{-1} \circ i,$ where $s
\in S$, $s^{-1}$ comes from $\hw[S\ob]_{add}.$ So, the composition comes from
$\hw[S\ob]_{add}.$

5. The universal properties of 
the additive localization construction (see Theorem \ref{taddlocu}) and
of the Karoubization one yield a canonical 
sequence of functors $\au[S\ob]_{add}
\to \hw[S\ob]_{add}\to \kar(\au[S\ob]_{add})\to \kar(\hw[S\ob]_{add})$.
 Together with  the previous assertion it implies that 
 $\au[S\ob]_{add}$
is a full subcategory of $\hw[S\ob]_{add}$. Thus, $\au[S\ob]_{add}$ is a
negative subcategory of $\cu/\du$. Lastly, by Proposition \ref{pbw}(7) there is
only one weight structure whose heart contains $\au[S\ob]_{add}$. 

\end{proof}

\begin{rema}\label{r2purloc}
 
1. Part 4 of the theorem above is not a generalization of Theorem
\ref{taddlocu} (in contrast to part 5), since $\au$ is not necessarily equivalent to the heart of
the ``stupid'' weight structure (either for $K^b(\au)$ or for $K(\au)$; see Proposition
\ref{pbw}(\ref{igen}).

 2. For $\cu$ possessing a differential graded enhancement (``compatible with $w$'')  one can verify 
 the existence of $w_{\cu/\du}$ ``explicitly''; cf.  \S8.2 of \cite{mymot}.
 Yet it seems hard to compute $\hw_{\cu/\du}$ via calculations of this sort.
 
 3. Our theorem 
 extends the results of \S8.1 of \cite{bws}, where (essentially) the case
 $B\subset \cu_{w=0}$ was considered (see 
 Remark \ref{rintro} above).

 Note that our setting reduces to this one if all the elements of $S$ are
 retractions. 

\end{rema}

\subsection{The relation with adjacent \texorpdfstring{$t$}{t}-structures}\label{sadja} 

In this section we discuss the relation of our methods with the 
``triangulated'' approach to localizations used in \cite{dwy} and in
\cite{neeran}  (that is closely connected with $t$-structures).
First, we demonstrate that weight structures can be used for construction and
study of certain $t$-structures for a wide class of triangulated theories (and
their localizations). 
 Next, we explain the relation of the methods used in  ibid. with
our ones. %

The idea is to find a weight-exact embedding 
  of a 'small' triangulated category $\cu'$ into a certain 'big' $\cu$ (with a weight
  structure $w$) such that $w$ is dual to a certain $t$-structure (in a certain
  sense). Then the weight-exactness of the localization functor becomes
  equivalent to the $t$-exactness of its right adjoint.

For simplicity, 
till the end of the section we will assume that $\cu$ contains 
arbitrary small coproducts of its objects (below
for a category $C$ fulfilling this condition we will just say that $C$ is closed with respect to coproducts)
everywhere except 
assertions I and II of
Proposition \ref{padstr}.
Note yet that for our purposes it is also possible to
consider  coproducts for sets of objects of limited cardinality only; cf.
Theorem 4.3.2(III(i)) of \cite{bws} . In such a situation an $X\in \obj \cu$ is called
compact if the functor $\cu(X,-)$ commutes with all possible coproducts.

\subsubsection{On adjacent structures and localizations}
 

We will need some definitions and notation.

Let $\cu$ be a triangulated category closed with respect to  
coproducts, $B\subset \obj \cu$. Then we will call a 
subcategory
$\du\subset \cu$ the {\it localizing subcategory generated by} $B$ (or that $B$ generates $\du$ as a localizing subcategory) if $\du$ is the smallest full strict triangulated subcategory of $\cu$ containing $B$ and closed with respect to  
coproducts. Note that in this situation $\du$ is called compactly generated if all 
elements of $B$ are compact in it; yet we will not really need the latter definition below.

For a small additive category $\au$ 
 we denote by $\adfu(\au,\ab)$ the
category of additive functors from $\au$ to the category of abelian groups.

Recall that $t$-structures were defined (in \cite{bbd}) in terms of certain
subclasses $\cu^{t\le 0},\cu^{t\ge 0}\subset \obj \cu$. These classes of
objects should satisfy certain axioms that are somewhat similar to Definition
\ref{dwstr} (this is certainly not a pure coincidence; ibid. inspired the
writing both of \cite{bws} and of \cite{konk}). The heart $\hrt$ of $t$ is
defined similarly to $\hw$.

We will say that a $t$-structure $t$ (for $\cu$) is non-degenerate if
$\cap_{n\in \z}\cu^{t\le n}= \cap_{n\in \z}\cu^{t\ge n}=\ns$.

\begin{defi}\label{deadj} We say that a weight structure $w$ for $\cu$ is
{\it adjacent} to a $t$-structure $t$ for $\cu$ if and only if the class
$\cu_{w\ge 0}$ coincides with $\cu^{t\le 0}$.

\end{defi}

\begin{rema}
This definition  (along with Proposition \ref{padstr}(I,II))
 axiomatizes the ``duality'' between the category $\au$ of 
 projective $R$-modules and the category $\aau$ of all $R$-modules. The
 ``triangulated avatar'' of this duality is the fact that any object of
 $D(R)^{t\le 0}$ has a projective resolution ($t$ is the canonical
 $t$-structure for $\cu=D^-(R)$); so, the    ``stupid'' weight structure $w$
 coming from the isomorphism $\cu\cong K^-(\au)$  is adjacent to $t$. 
Note however: in this example we have $\hw\subset \hrt$; yet we do not have
such an inclusion for a general pair of adjacent $w,t$.
\end{rema}

We will  need the following properties of adjacent structures and compactly generated categories; most of them were proved in \S\S4.4-4.5 of \cite{bws} (yet pay attention to Remark \ref{rstws} above!).

\begin{pr}\label{padstr}

I Assume that $\cu$
is  endowed with a
weight structure $w$  and also with an adjacent
$t$-structure $t$ (we do not assume that 
$\cu$ is closed with respect to coproducts in this assertion).

1. Assume that $\cu$ is a small category. Denote by
$\cu(\hw,-)$ the functor $\hrt \to \adfu(\hw^{op}, \ab)$ that sends $N\in
\cu^{t=0}$ to the representable functor $M\mapsto \cu(M,N)\ (M\in \cu_{w=0}).$
It is an exact embedding of $\hrt$ into the abelian category $\adfu(\hw^{op}, \ab)$.

2. Assume that $t$ is non-degenerate. Then
$\cu^{t\le 0}=(\cup_{i<0}\cu_{w=i})^{\perp}$, 
$\cu^{t\ge 0}=(\cup_{i>0}\cu_{w=i})^{\perp}$.

II Assume that $\cu$
is  endowed with a
weight structure $w$  and  with an adjacent
$t$-structure $t$ (and not necessarily closed with respect to coproducts), whereas  a triangulated category $\eeu$
is endowed with adjacent
 $w_{\eeu}$ and $ t_{\eeu}$. Let $\pi:\cu\to \eeu$
 be an exact functor. Then the following statements are valid. 

1. $\pi$ is right weight-exact if and only if it is 
right  $t$-exact (i.e., if $\pi(\cu^{t\le 0})\subset \eeu^{t_{\eeu}\le 0}$).

2. Let $G:\eeu\to \cu$ be the right adjoint to $\pi$. Then $\pi$ is right
(resp. left) weight-exact with respect to $w$ and $w_{\eeu}$ if and only if $G$
is left (resp. right) 
$t$-exact with respect to $t_{\eeu}$ and $t$.

III Let $\cu'$ be a full triangulated subcategory of $\cu$ 
whose objects are compact; 
 let $B\subset \obj\cu'$
be 
a 
set (i.e., it is small); 
let $\hu\subset \cu'$ be a small additive category such that $\obj \hu$ 
generates $\cu$ as (its own) localizing subcategory (whereas $\cu$ is closed with respect to coproducts). Denote by $\ddu$ the localizing subcategory of $\cu$ generated by $B$.
 Then the following statements are valid. 

1. The Verdier localization category $\eeu=\cu/\ddu$ exists;
the localization
functor $\pi:\cu\to \eeu$ commutes with  coproducts, converts
compact objects into compact ones, and possesses a right adjoint $G$ that is a full
embedding functor. Besides, $\eeu$ is generated by $\pi(\obj \hu)$ as a localizing subcategory. 

2. $\pi$ induces a full embedding of $\eu=\cu'/\lan B\ra_{\cu'} $ into $ \eeu$.

3. Assume moreover that $\hu$ is negative (see Definition
\ref{dwstr}(VII)).
Then $\cu$ possesses a weight structure $w$ 
whose heart is equivalent to the Karoubization of the category of all
the coproducts of the objects of $\hu$, and also a non-degenerate $t$-structure
$t$ such that $w$ is adjacent to $t$.
Besides, $\hw$ is the Karoubization of the category of all ``formal'' (small)
coproducts of objects of $\hu$ (i.e., we have $\cu(\coprod_{l\in L}H_{l},\coprod_{j\in J}H_{j})
=\prod_{l\in L}(\bigoplus_{j\in J}\cu(H_{l},H_{j}))$;
here $H_{l}, H_j\in \obj \hu$, $L,J$ are index sets), whereas 
$\hrt\cong \adfu (\hu^{op},\ab)$  (via the functor $N\mapsto (H\in \obj
\hu\mapsto \cu(H,N))$).

\end{pr}
\begin{proof}

I Assertion 1 is exactly part 4 of Theorem 4.4.2 of \cite{bws}; it easily
implies assertion 2 (which one can reduce to the case of a small $\cu$).

II.1. Immediate from the definition of adjacent structures.

2. See Remark 4.4.6 of ibid.

III 1--2: These statements easily follow from the results of \cite{neebook}.

  Next, Theorem 8.3.3 of \cite{neebook} implies that $\du$ 
   satisfies the Brown
 representability condition (see Definition 8.2.1 of ibid.).
 Hence  Proposition 9.1.19 of ibid. yields the existence of $\eeu$ and 
  $G$;  
 $G$ is a full embedding since it is adjoint to a localization
functor. Corollary 3.2.11 of  \cite{neebook} yields that $\pi$ respects
 coproducts;
 this
 finishes the proof of
 assertion III.1.   Assertion III.2 is given by Theorem 4.4.9 of ibid. (one should take
$\alpha=\aleph_0$ in it).

3. Let $c$ be a non-zero object of $\cu$. Then there exist  $h\in \obj \hu$, $i\in \z$, and $m\in \cu(h[i],c)$. Indeed, it suffices to note that the full 
subcategory of $\cu$ whose objects are ${}^\perp\{c[-i],\ i\in\z\}$ is  triangulated, strict, and   is closed with respect to 
 coproducts; hence it can contain $\hu$ only if $c=0$.  
Hence  $\cu$ is {\it negatively well-generated} by $\obj \hu$ in the sense of Definition 4.5.1 of \cite{bws}. Thus  we can apply Theorem 4.5.2 of  ibid.
 to $\cu$; this yields the result.

\end{proof}

One may say that part II.2 of this Proposition settles the difficulties caused
by the fact that the (left) adjoint to a $t$-exact functor is usually not
$t$-exact.

Now we formulate the main result of this subsection; we will consider an example
for it in detail in \S\ref{scompneera}.

\begin{theo}\label{tadjnew}

Let $\cu,\cu',B,\hu,\ddu,\eeu,\pi,G,
\eu$ be as in Proposition \ref{padstr}(III)
(and satisfy all the conditions
mentioned in loc. cit.). Assume moreover that $B$ consists of cones of some
set  $S\subset \mo \hu$, and that $\hu$ generates $\cu'$ (see \S\ref{snotata}). Then the following
statements are valid.

1.  $\eeu$ possesses a 
 weight structure $w_{\eeu}$ and also a non-degenerate
adjacent $t$-structure $t_{\eeu}$.

2. $\hu[S\ob]_{add}$ is a full additive subcategory of $\hw_{\eeu}$ (via
$\pi$); $\hw_{\eeu}$ is equivalent to the Karoubization of the category of all
small coproducts of objects of $\hu[S\ob]_{add}$.

3. 
$w_{\eeu}$ restricts to a weight structure for $\eu\subset \eeu$. 

4. $G$ is $t$-exact; $\pi$ is weight-exact.

5. $\hrt_{\eeu}$ is equivalent to $\adfu(\hu[S\ob]_{add}^{op},\ab)$. It is 
a full exact abelian subcategory of $\hrt\cong \adfu (\hu^{op},\ab)$ (via $G$); 
an $I\in \obj \adfu (\hu^{op},\ab)$ belongs to  $G(\hrt_{\eeu})$ if and only if $I(s)$ is  bijective
for any $s\in S$.

\end{theo}
\begin{proof}

Note that  the full subcategory of $\eu$ 
 whose objects are those of $\hu$ is isomorphic to  
$\hu[S\ob]_{add}$ (this is the first part of our assertion 2), and  is negative
in $\eu\subset \eeu$.
Hence
Proposition \ref{padstr}(III) yields that 
$\eeu$ satisfies the conditions of Proposition \ref{padstr}(III.3) with
$\hu[S\ob]_{add}$ instead of $\hu$. Thus our assertion 1, the 
rest of assertion 2, and the first part of assertion 5 follow from loc. cit.

Now, 
Theorem \ref{twloc}(5) also yields the existence of a weight structure for $\eu$ whose heart is $\kar_{\eu}(\hu[S\ob]_{add})$;
it is compatible with $w_{\eeu}$
 by  
Proposition \ref{pbw}(\ref{iwe}). Thus we proved assertion 3.

Next, since $\pi$ commutes with  coproducts, we have
$\pi(\cu_{w=0})\subset \eeu_{w_{\eeu}=0}$. Hence Proposition \ref{padstr}(I.2)
yields the $t$-exactness of $G$. Part II.2 of loc. cit. also implies that
$\pi$ is weight-exact; this finishes the proof of assertion 4.

The second part of assertion 5 is 
given by the $t$-exactness of $G$. In order to verify its last part, it suffices
to note that an $I\in \adfu(\hu^{op},\ab)$ factors through $\hu[S\ob]_{add}$
if and only if it converts all $s\in S$ into bijections.

\end{proof}

\begin{rema}\label{radcomp1}
1. It is easy to construct such a setting for any 
small $\hu$. To this end, following \S4 of \cite{dwy} one can 
take for $\cu$  the derived category of $\adfu(\hu^{op},\ab)$. Note that we have $\hw\subset \hrt$ in this example (yet we don't have an inclusion of $\hw_{\eeu}$ into $\hrt_{\eeu}$ for a general $S$ since $\pi$ is not $t$-exact).


2.  
Any object of $\hu$ yields the obvious exact functor from the category
$\adfu(\hu^{op},\ab)\cong \hrt$ to $\ab$. Set-theoretic difficulties prevent
us from saying that $\hu^{op}$ embeds into the category of exact functors from
$\hrt$ to $\ab$ (and the same is true for $\hu[S\ob]_{add}^{op}$ and
$\hrt_{\eeu}$).  Yet one can still choose a (large enough) small abelian
(exact) subcategory $\aau\subset \hrt$ such that $\hu^{op}$ embeds into the
category of exact functors $\aau\to\ab$. So we obtain certain ``Yoneda-like''
descriptions of $\hu$ and $\hu[S\ob]_{add}$. Note also that the restriction of
$\pi$ to $\hu$ can be described via the restriction of the exact functors
$\aau\to\ab$ mentioned to the subcategory $\aau\cap G(\hrt_{\eeu})$.

One can also avoid set-theoretic difficulties by finding some ``small
substitute'' for $\cu$; cf. Definition 4.7 of \cite{neeran}.

3. One can also prove the existence of $w_{\eeu}$ and the weight-exactness of
$\pi$ more directly 
via an ``unbounded'' generalization of Proposition \ref{ploc}.
\end{rema}

\subsubsection{The relation between our results and 
some of the results of \texorpdfstring{\cite{dwy} and
\cite{neeran}}{those}}\label{scompneera}
 
We adopt the notation of Theorem \ref{tadjnew}.
 
The papers  \cite{dwy} and  \cite{neeran} mostly consider $\hu$ being the category of  finitely generated projective modules over a ring $R$ (though in \S4 of \cite{dwy} also arbitrary small additive categories were considered; cf. Remark \ref{radcomp1}(1)).
 Then one can take $\cu=D(R)$. The main ``localization'' results of the papers cited can be stated as follows: 
    the 
non-commutative localization ring $R[S\ob]$ is naturally isomorphic to the endomorphism ring of $\rr$ (i.e.,  of $R$ considered as a left $R$-module) in 
the  corresponding localization $\eeu$ of $D(R)$. Certainly, this statement also follows from our results (use the embedding $\eu\to \eeu$ and apply Theorem \ref{twloc}(4); cf. also \S\ref{sncloc}). Besides, note  that $\cu=D(R)$ is a very partial case of Theorem \ref{tadjnew}. 

\begin{rema}\label{radcomp2}

1. A major advantage of our methods is that we can use $\cu'=K^b(\hu)\subset D(R)$ instead of $\cu$ (and $\eu=\cu'/\lan \co(S)\ra$ instead of $\eeu$). Certainly, $\cu'$ is ``much smaller'' than $\cu$, whereas 
Proposition \ref{paddlocu} makes explicit computations in $\eu$ quite easy. Note in contrast that no attempt was made in \cite{dwy} and  \cite{neeran} to
``calculate $R[S\ob]$ explicitly'' 
(i.e., to reprove the results of \cite{gera} and \cite{malc}).
 Besides, weight structures (coming from {\it negative well-generating} subcategories of compact objects; see the proof of Proposition \ref{padstr}(III.3)) are very useful for constructing $t$-structures in this setting 
 cf. \S\S4.5--4.6 of \cite{bws}.
Lastly, to the belief of the authors, it is easier to deal with the functor $K^b(F)$ for an additive $F:\hu\to \au$ than with functors of the type $-\otimes_R R'$ for a ring homomorphism $R\to R'$.

2. The $t$-structure for $\cu$ given by 
Proposition \ref{padstr}(III.3) is exactly the canonical $t$-structure for $D(R)$. Thus for this particular setting
the existence of $t_{\eeu}$ (see Theorem \ref{tadjnew}(1)) compatible with the  embedding  $G:\eeu\to \cu$  was proved in Proposition 3.4 of \cite{dwy} and in Lemma 6.3 of \cite{neeran} (and was important for these papers).

The usage of $t$-structures 
explains why ``big'' triangulated categories were needed in 
the papers cited; cf. Remark 4.3.4(4) of \cite{bws}. 
Note in contrast that the weight-exactness of  $\pi$ holds for its restriction to a ``much smaller'' $\cu'$ (such a category was very important for the $K$-theory computations of \cite{neeran}).
Besides, the ``duality'' between weight and $t$-structures allows us to study the properties of $\pi$ (that is not left $t$-exact in general).

3. Certainly,  the $t$-exactness of  $G$  yields the right $t$-exactness of
$G\pi$. In Proposition 3.2 of \cite{dwy} it was verified: 
 $\cu(\rr,
G\pi(\rr))\cong \cu(\rr, H_0^t(G\pi(\rr)))$ 
is isomorphic to $R[S\ob]$. Certainly, this statement also follows from our
calculation of $\hw_{\eeu}$ (since $R[S\ob]\cong  \enom_{\eeu}\rr$).

4. In \cite{neeran}  the canonical $t$-structure for $D(R)$ was mentioned explicitly, whereas in \cite{dwy} the corresponding cohomology (denoted by $\pi_i$) was considered. 
The existence of cohomology functors is certainly an advantage of derived categories (and of general triangulated categories 
endowed with $t$-structures). Yet the results of the current paper (together with the ones of \S2.5--2.6 of \cite{bger}) demonstrate that in certain situations it is quite useful to consider a certain weight structure instead of a $t$-structure {\it orthogonal} to it (see Definition 2.5.1 of ibid.; this is a certain generalization of Definition \ref{deadj} above).

5. For a triangulated category $\cu$ possessing a differential graded
enhancement
(see Definition 6.1.2(3) of \cite{bws}; this includes all possible triangulated
categories whose objects are various complexes) and a weight structure $w$ on
it the weak weight complex functor (see Proposition \ref{pbw}(\ref{iwwc})) can
be lifted to an exact functor $t^{st}:\cu\to K(\hw)$ (see \S6.3 of \cite{bws}
in the case when $w$ is bounded). So, the localization homomorphism $R\to
R[S\ob]$ is {\it stably} flat (see Theorem 0.7 of \cite{neeran} and
Proposition 3.3 of \cite{dwy}) if and only if $t^{st}_{\eu}$ is fully
faithful.

6. Certainly, in \cite{neeran} and \cite{dwy} (and in several related papers of
the same authors) there are several results that we do not mention above (since
they are beyond the scope of the current paper).
They mostly concern the case when the localization homomorphism $R\to R[S\ob]$
is stably flat; also, several examples of non-commutative localizations were
considered.

\end{rema}

\section{On birational motives and weights for them}\label{sbir}

We prove that our results yield an ``elementary'' proof of the existence of
certain (birational) weight structures for various triangulated categories of birational
motives (defined via the method of \cite{kabir}). Besides, Theorem \ref{tadjnew}
yields the existence of  $t$-structures adjacent to these weight structures; in ibid. in the case of 
motives over a field it was shown that the corresponding $t$-structure is a restriction of  the Voevodsky's homotopy $t$-structure to the category of birational motivic complexes.

\subsection{The definition of birational motives}\label{sbmot}

Now we define certain categories of birational motives over a base $U$. We
generalize 
the corresponding definition from \cite{kabir} to a somewhat more general context.

For a scheme $U$ one considers a certain additive category $\cors$. Two
important 'types' of $\cors$ were described in detail in  \cite{degcis}
starting from \S9 (there $U$ was assumed to be 
noetherian separated). We do not need a precise definition of $\cors$ in this
paper (so, we will not specify which version of it we consider); in particular,
one can choose arbitrary (associative commutative unital) ``coefficient rings''
when defining $\cors$. We will only require $\cors$ to satisfy the following
properties:  its objects are certain finite type 
 $U$-schemes (so, the category is essentially small); the disjoint union
 operation  yields the direct sum in $\cors$; for any $X/U$ such that $X\in
 \obj \cors$ and any smooth finite type $f:Y\to X$ we have: $Y\in \obj \cors$
 and $f\in \cors(Y,X)$. We will make no distinction in notation between schemes
 and their morphisms and the corresponding objects and morphisms in various
 ``motivic'' categories that we consider in this section.

Next, (see \S11.1.9 of ibid.) one considers $K^b(\cors)$ and (following
\cite{flo}) localizes it by the 
triangulated subcategory generated by 
two types of complexes: the complex $\af^1\times X\to X$ for any $X\in \obj
\cors$ (we will denote the set of these complexes by $B_{HI}$) 
and by
complexes of the form $W\xrightarrow{\begin{pmatrix}-k \\ g \end{pmatrix}}
Y\bigoplus V\xrightarrow{\begin{pmatrix} f & j \end{pmatrix}} X$ 
(we will  
denote the set of these complexes by $B_{MV}$)
for any elementary (Nisnevich) distinguished square
  $$\begin{CD}
 W@>{k}>>Y\\
@VV{g}V@VV{f}V \\
V@>{j}>>X
\end{CD}
$$
of objects of $\cors$ (this is a certain Cartesian square; $j,k$ are open dense
embeddings).
One defines $\dmges$ as the Karoubization of the triangulated category
obtained.

Now (following Definition 5.1 of \cite{kabir}) we define $\birs$ as the
Karoubi\-zation of the localization $\dmges$ by the subcategory generated by
(images in $\dmges$ of) complexes of the form  $U\stackrel{j}{\to} X$ 
 for all open dense embeddings $j$ of objects of $\cors$ (here the upper index ``o'' stands for ``open''). We will denote the
 set of these complexes by $B_{bir}$.

\subsection{A weight structure on \texorpdfstring{$\birs$}{birational geometric
motives}}

We make the following simple observation: by definition, 
$B_{MV}\subset \lan B_{bir} \ra$. Hence $\birs$ is (isomorphic to) the
Karoubization of the localization of $K^b(\cors)$ by 
$\lan B_{bir}\cup B_{HI}\ra$.  Hence we can apply all the results of the
previous sections to this setting!

\begin{theo}\label{wsbir}
1. The Karoubi-closures of the 
subclasses $(K^b(\cors)^{\le
0}$, $K^b(\cors)^{\ge 0}\subset \obj \birs)$ yield a weight structure $w_{bir}$
on $\birs$. Moreover, this is  the unique weight structure such that the localization functor
$K^b(\cors) \to \birs$ is exact.

2. Denote the heart of $w_{bir}$ by $\birchs$. Its objects are exactly all the retracts
of
$\obj \cors$
 in $\birs$, and it is isomorphic to the Karoubization of
$\cors[(S_{bir}\cup S_{HI})\ob]_{add} \cong \cors[S_{bir}
\ob]_{add} \cong \cors[S_{bir}\ob]$. 
\end{theo}
\begin{proof}
1. The existence follows from the Theorem \ref{twloc}
and the uniqueness can be obtained from 
Theorem \ref{twloc}(5).
2. The isomorphism of  $\birchs$ with $\kar(\cors[(S_{bir}\cup S_{HI})\ob]_{add})$ follows from
\ref{twloc}(4). Two remaining isomorphisms can be proved via 
an argument of J.-L.
Colliot-Th\'el\'ene (see Appendix A of \cite{kabiratgeom}) using the fact that
$S_{bir}$ is closed under direct sums (see Remark \ref{rcismo}(1)).
\end{proof}

 
Now we describe certain ways 
of making computations in $\birs$ and
$\birchs$.


  It could make sense to  start from the localization of $K^b(\cors)$ by $\lan
  B_{HI}\ra$. Note here that all elements of $S_{HI}$ are coretractions (cf.
  Remark \ref{r2purloc}(3)); so, one can apply the calculations from \S8.2 of
  \cite{mymot}. Moreover, one can probably express $D_{HI}=K^b(\cors)/\lan
  B_{HI}\ra$ in terms of certain (cubical) Suslin complexes; cf. \S5 of ibid.
  In particular, the morphisms between objects of $\cors$ in $D_{HI}$ could be
  expressed as the morphisms in $\cors$ modulo certain ``homotopy equivalence''
  relation.
  Next, one should localize $D_{HI}$ by $B_{bir}$ (and consider the
  Karoubization of the category obtained).
 Note here:  in order to compute $\birchs$, it suffices to consider the
 localization of the $K^b(\hw_{HI})$ by $B_{bir}$ (here we apply 
 Theorem \ref{twloc}(5)).

One can also apply Remark \ref{radcomp1}(2) to  calculations in $\birchs$
(somehow).
Note here: Theorem \ref{tadjnew}(1,4,5) yields the natural generalization of
Proposition 7.4 
 of \cite{kabir} in our setting 
(the corresponding $\hrt_{\eeu}$ was 
denoted by $HI^0(-)$ in 
  \cite{kabir}). 

Now we briefly recall some consequences of the existence of a bounded weight
structure for $\birs$.

\begin{rema}

1. There exists an exact conservative {\it weight complex} functor $\birs\to
K^b(\birchs)$; cf. \S3.1 of \cite{brelmot}.

2. We define $K_0(\birchs)$ as the abelian group with a generator $[X]$
for each $X\in \obj \birchs$, and the relation $[Y]=[X]+[Z]$ for any
$X,Y,Z\in\obj \birchs$ such that $Y\cong X\bigoplus Z$.
For $K_0(\birs)$ 
we take similar generators and  set $[Y]=[X]+[Z]$ if
$X\to Y\to Z\to X[1]$ is a distinguished triangle.
Then the embedding $\birchs\to \birs$  yields an
isomorphism $K_0(\birchs)\cong K_0(\birs)$; see Theorem 5.3.1 of \cite{bws}.

3. For any (co)homological functor $H:\birs\to \au$ ($\au$ is an abelian
category) there exists certain {\it Chow-weight} spectral sequences $T(H,X)$
that relate the cohomology of any $X\in \obj  \birs$ with that of the terms of
its weight complex. $T(H,X)$ is $\birs$-functorial in $X$ starting from $E_2$.
 $T(H,X)$  induces a certain (Chow)-weight filtration on $H_*(X)$ (or
 $H^*(X)$); this filtration is also $\birs$-functorial and can be (easily)
 described in terms of weight decompositions (only); see \S2 of ibid.


\end{rema}

We also make certain remarks on the relevance of our results.

\begin{rema}\label{rsplit}

1. So we obtain 
some new tools for ``computing'' $\birchs$ (for a general $U$).  
Moreover, the authors hope that Remark \ref{rcompu}(1) 
 will be relevant for computing  $\birs$-morphism groups between 
$\birs_{w_{bir}=0}$ and $\birs_{w_{bir}=n}$ for (certain) $n<0$.

Another relevant observation here: for any additive category $\au,\ S\subset \mo(\au)$,
$\cu=K^b(\au)$, $B=\co(S)\subset K^b(\au)^{[-1,0]}$, one can easily describe 
$B_{[n,0]}$
(cf. Proposition \ref{ploc}(2,3)) by a natural
generalization of (\ref{etrmor}).


2. It is certainly ``more interesting'' to consider weight decompositions of
those objects of $\birs$ that do not belong to $\birs_{w=i}$ for any $i\in \z$
(then one can obtain non-trivial weight filtrations and weight spectral
sequences for (co)homology). Certainly, we would also like these objects to
``come from $U$-schemes''.
So, it seems natural to look for some birational motives of non-smooth finite
type $U$-schemes and
 for (birational) motives with compact support of (certain) finite type
 $U$-schemes.
 If one can define certain motives of this type for $U$-schemes in $\dmges$,
 then one can also obtain the corresponding birational motives by applying the
 localization functor. Unfortunately, it seems that for a general $U$ the
 theory of \cite{degcis} yields ``reasonable'' motives (with compact support)
 for arbitrary finite type $U$-schemes only in the category $\dmcs$ (this is
 the ``stabilization'' of $\dmges$ with respect to Tate twists; it's better to
 consider motives with rational coefficients here). Yet this is no problem 
 if $U$ is the spectrum of a perfect field. Besides,  in the general case one
 can still define certain motives (with compact support) of 
 $U$-schemes in $\dmes$ (this is a ``cocompletion'' of $\dmges$; one also
 obtains certain ``non-constructible'' birational motives this way) using
 the results of ibid.

3. Since an $X\in\obj \cors$ in $\birs$ becomes isomorphic to any open dense
$X'\subset X$, the objects of $\birchs$ are retracts of finite coproducts of ``birational
motives of the generic points of objects of $\cors$''. Moreover, we suspect (following
\cite{kap}) that $\birs$ splits as the direct sum of the corresponding
categories over (generic) Zariski points of $U$. Yet this could depend on the
version of $\cors$ chosen. In any case, even the existence of $w_{bir}$ was not
clear previously for $U$ being the spectrum of an imperfect field.


4. In the case when $U$ is the spectrum of a perfect field $k$ (and for the
``classical'', i.e., Voevodsky's $\cors$; see \cite{1}) the existence of
$w_{bir}$ was already known previously. One can distinguish three (``old'')
methods of the proof. Their (common) main disadvantage is that they rely on
certain (quite hard) results of Voevodsky; they also require certain additional
restrictions on $U$.

(i) By Proposition \ref{pbw}(\ref{igen},\ref{idemp}) it suffices to verify that
the objects of $\cors$ (i.e., smooth $k$-varieties) form a negative subcategory
of $\birs$. This was done in \S7 of \cite{kabir} in the case $\cha U=0$.

(ii) Another possible approach (that yields two distinct proofs of the existence of $w_{bir}$) is to prove the existence of some weight structure
$w_{eff}$ for  $\dmge=\dmge(k)$ or for some its ``completion'' that 
yields a weight structure for $\birs$. Note here: in this case the subcategory
$\lan B_{bir}\ra_{\dmge}$ is exactly the Tate twist $\dmge(1)$ of $\dmge$
(resp. of its ``completion''); it is isomorphic to $\dmge$ by the Voevodsky's
cancellation theorem. Hence by Theorem \ref{twloc}(1) 
or by  Proposition 8.1.1(1) of \cite{bws} (that is essentially its partial case;
see Remark \ref{r2purloc}(3) and Remark \ref{rintro}) $w_{eff}$ yields a weight 
structure for $\birs$ if there exists an $i\in \z$ such that the functor
$-(1)[i]$ (i.e., the composition of the Tate twist with $[i]$) is weight-exact
with respect to $w_{eff}$ (see Definition \ref{dwstr}(VI)).

Surprisingly, there are two ``geometric'' possibilities here and an infinite
set of certain ``less explicit'' ones; they are (``mostly'') parametrized by
$i$.

(iia) In the case $i=2$ there is the {\it Chow weight structure}; its heart is
the category of effective Chow motives (which is a full subcategory of
$\dmge$; see \cite{1}). This was our reason for denoting $\hw_{bir}$ by
$\birchs$.
 
In the case $\cha k=0$ one can take any coefficient ring here (see Proposition
6.5.3 of \cite{bws}); in the case $\cha k=p>0$ one has to invert $p$ (and apply
the main result of \cite{bmres}).

(iib) For $i=1$ 
one can construct a certain Gersten
weight structure. Its heart is ``cogenerated''  by certain (co)motives of function
fields over $U$ (see Proposition 4.1.1 of \cite{bger} for more detail in the
case when $k$ is countable, and \S6.4 of \cite{bgern} for the general case); so,
it is only defined for a certain ``completion'' $\gd(k)$ of $ \dmge(k)$.

(iic) One can also construct a ``$-(1)[i]$-stable'' weight structure on a
certain version of $\gd(k)$ for any $i\in \z$; see \S4.9 of \cite{bger}. Even
more generally, one can ``glue'' certain weight structure from $w_{bir}$ by
considering certain ``shifts'' of it (by arbitrary integers)
 on the
categories $\dmge(j)/\dmge(j+1)\cong \birs$ for $j$ running through all $\n$.
Yet the author does not know of any constructions of  weight structures of this
sort that would not rely on the existence of the  ``explicit'' weight
structures (that were already mentioned above).

\end{rema}

\end{document}